\documentclass[draft]{siamltex} 	

\usepackage{amsmath,amssymb,cite,mathtools}

\newcommand{\Diag}[1]{\textup{Diag}{(#1)}}
\newcommand{\matr}[2]{[#1\ \dots\ #2]}
\newcommand{\condWmC}{$\text{\textup{(W{\scriptsize m}}}_{\mathbf C}$$\text{\textup{)}}$}
\newcommand{\condWm}{$\text{\textup{(W{\scriptsize m}}}$$\text{\textup{)}}$}
\newcommand{\condUm}{$\text{\textup{(U{\scriptsize m}}}$$\text{\textup{)}}$}
\newcommand{\condCm}{$\text{\textup{(C{\scriptsize m}}}$$\text{\textup{)}}$}
\newcommand{\condKm}{$\text{\textup{(K{\scriptsize m}}}$$\text{\textup{)}}$}
\newcommand{\condUmC}{$\text{\textup{(U{\scriptsize m}}}_{\mathbf C}$$\text{\textup{)}}$}
\newcommand{\condUmA}{$\text{\textup{(U{\scriptsize m}}}_{\mathbf A}$$\text{\textup{)}}$}

\newtheorem{example}[theorem]{Example}
\newtheorem{procedure}[theorem]{Procedure}

\title{Generic uniqueness conditions for the canonical polyadic decomposition 
and INDSCAL
\thanks{Research supported by: (1) Research Council KU Leuven: C1 project c16/15/059-nD, GOA/10/09 MaNet, CoE PFV/10/002 (OPTEC), PDM postdoc grant; (2) F.W.O.:  project  G.0830.14N, G.0881.14N;  (3) the Belgian Federal Science Policy Office: IUAP P7 (DYSCO II,  Dynamical systems, control
and optimization,  2012-2017); (4)  EU: The research leading to these results has received funding from the European Research Council under the European Union's Seventh Framework Programme (FP7/2007-2013) / ERC Advanced Grant: BIOTENSORS (no.  339804). This paper reflects only the authors' views and the Union is not liable for any use that may be made of the contained information.
}}
\author{Ignat Domanov\footnotemark[2] \footnotemark[3] \footnotemark[4]
\and
Lieven De Lathauwer\footnotemark[2] \footnotemark[3] \footnotemark[4]}
\begin{document}
\maketitle
\renewcommand{\thefootnote}{\fnsymbol{footnote}}

\footnotetext[2]{Group Science, Engineering and Technology, KU Leuven - Kulak,
E. Sabbelaan 53, 8500 Kortrijk, Belgium
 ({\tt ignat.domanov,\ lieven.delathauwer@kuleuven-kulak.be}).}
\footnotetext[3]{Dept. of Electrical Engineering  ESAT/STADIUS KU Leuven,
Kasteelpark Arenberg 10, bus 2446, B-3001 Leuven-Heverlee, Belgium.}
\footnotetext[4] {iMinds Medical IT.}

\begin{abstract}
We find conditions that guarantee that a decomposition of a generic third-order tensor in a minimal number of rank-$1$ tensors
(canonical polyadic decomposition (CPD)) is unique up to permutation of rank-$1$ tensors.
Then we consider the case when the tensor and all its rank-$1$ terms have symmetric frontal slices (INDSCAL).
Our results  complement  the existing  bounds for generic uniqueness of the CPD and relax the existing bounds for
INDSCAL. The derivation makes use of algebraic geometry. We stress  the power of the underlying concepts for proving generic properties in mathematical engineering.
\end{abstract}



\pagestyle{myheadings}
\thispagestyle{plain}
\markboth{IGNAT DOMANOV AND LIEVEN DE LATHAUWER}{Generic uniqueness conditions for CPD and INDSCAL}

\section{Introduction}
\subsection{Basic definitions}
Throughout the paper $\mathbb F$ denotes the field of real or complex numbers.
A  tensor $\mathcal T=(t_{ijk})\in\mathbb F^{I\times J\times K}$ is {\em rank-$1$} 
if there exist  three nonzero vectors $\mathbf a\in\mathbb F^I$,
$\mathbf b\in\mathbb F^J$ and $\mathbf c\in\mathbb F^K$ such that
$\mathcal T=\mathbf a\circ\mathbf b\circ\mathbf c$,
in  which ``$\circ$'' denotes the {\em outer product}. That is,  $t_{ijk} = a_i b_j c_k$ for all values of the indices.
A {\em  polyadic  decomposition} (PD) of a third-order tensor $\mathcal T$ expresses $\mathcal T$ as a  sum of rank-$1$ terms,
\begin{equation}
\mathcal T=\sum\limits_{r=1}^R\mathbf a_r\circ \mathbf b_r\circ \mathbf c_r,
 \label{eqintro2}
\end{equation}
where
$\mathbf a_r \in \mathbb F^{I}$, $\mathbf b_r \in \mathbb F^{J}$, $\mathbf c_r \in \mathbb F^{K}$ are nonzero vectors.
We will write \eqref{eqintro2}  as $\mathcal T=[\mathbf A,\mathbf B,\mathbf C]_R$,
where
$
\mathbf A =[\mathbf a_1\ \dots\ \mathbf a_R] \in\mathbb F^{I\times R},\quad
\mathbf B =[\mathbf b_1\ \dots\ \mathbf b_R]\in\mathbb F^{J\times R},\quad
\mathbf C =[\mathbf c_1\ \dots\ \mathbf c_R]\in\mathbb F^{K\times R}.
$

If the number $R$ in \eqref{eqintro2} is minimal, then it is called {\em the rank of $\mathcal T$} and is denoted by $r_{\mathcal T}$.
In this case we say that \eqref{eqintro2} is a {\em canonical polyadic decomposition} (CPD) of $\mathcal T$.
The CPD was introduced by F.L. Hitchcock in \cite{Hitchcock}. It is also  referred to as rank decomposition,
Canonical Decomposition (Candecomp) \cite{1970_Carroll_Chang}, and  Parallel Factor Model (Parafac) \cite{Harshman1970,1994HarshmanLundy}.

It is clear that in \eqref{eqintro2}  the rank-$1$ terms can be
arbitrarily permuted and that vectors within the same rank-$1$ term can be
arbitrarily scaled provided the overall rank-$1$ term remains the same. 
{\em The CPD of a tensor  is  unique} when it is only subject to these trivial indeterminacies.

We call tensors  whose  frontal slices are symmetric matrices (implying $I=J$)  SFS-tensors, where the abbreviation
``SFS'' stands for ``symmetric frontal slices''. 
It is clear that if an SFS tensor $\mathcal T$ is rank-$1$, then $\mathcal T=\mathbf a\circ\mathbf a\circ\mathbf c$ for some nonzero vectors
$\mathbf a \in \mathbb F^{I}$ and $\mathbf c \in \mathbb F^{K}$. 
Similarly to the unstructured case above, one can easily 
define the  SFS-rank, the SFS-CPD, and  the uniqueness of the SFS-CPD
of an SFS-tensor $\mathcal T$ (see \cite[Section 4]{PartII} for the exact definitions).
Note that the SFS-CPD corresponds to the INdividual Differences in multidimensional SCALing (INDSCAL) model, as
		 introduced by Carroll and Chang \cite{1970_Carroll_Chang}.
To the authors' knowledge, it is still an open question whether
there exist SFS-tensors with unique SFS-CPD but non-unique CPD.

Blind signal separation (BSS) consists of the splitting of
signals into meaningful, interpretable components. 
The CPD  has become a  standard tool for BSS:
 the known mixture of signals corresponds to a given tensor $\mathcal T$ and 
 the unknown interpretable components correspond to the rank-$1$ terms in its CPD.
 For the interpretation of the components one should be able to assess whether the CPD is unique.
 The SFS-CPD is a constrained version of the CPD. 
 In the original formulation of the INDSCAL model (or SFS-CPD) the frontal slices of $\mathcal T$ were distance matrices.
  Nowadays, \mbox{SFS-CPD}  is widely used in  independent component analysis (ICA) where the frontal slices of $\mathcal T$ are spatial covariance matrices. The SFS-CPD interpretation of ICA allows one to handle the underdetermined case (more sources than sensors).
        The (SFS-)CPD   based approach 
has found many applications in signal processing \cite{Cichocki2009},\cite{ComoJ10},
data analysis \cite{Kroonenberg2008}, chemometrics \cite{smilde2004multi},
psychometrics \cite{1970_Carroll_Chang}, etc.
We refer the readers to the overview papers  \cite{KoldaReview,2009Comonetall,Lieven_ISPA, Sorber,LievenCichocki2013}
and the references therein for background,  applications and algorithms.

The most famous result on uniqueness is due to J. Kruskal \cite{Kruskal1977}. The $k$-rank of a matrix $\mathbf A$ is defined as the largest integer $k_{\mathbf A}$ such that
any $k_{\mathbf A}$ columns of $\mathbf A$ are linearly independent. Kruskal's theorem states that if $\mathcal T=[\mathbf A,\mathbf B,\mathbf C]_R$ and
\begin{equation}
 R\leq \frac{k_{\mathbf A}+k_{\mathbf B}+k_{\mathbf C}-2}{2},\label{eqkruskal}
\end{equation}
then $r_{\mathcal T}=R$ and the CPD of $\mathcal T$ is unique.

Condition \eqref{eqkruskal} is an example of a deterministic condition for uniqueness in the sense that the uniqueness of the CPD can be
guaranteed for a particular choice of the matrices $\mathbf A$, $\mathbf B$, and $\mathbf C$.
Checking deterministic conditions can be cumbersome. For instance, in \eqref{eqkruskal} the computation of the $k$-ranks has combinatorial complexity.
If the entries of  matrices $\mathbf A$, $\mathbf B$, and $\mathbf C$ are  drawn
from continuous distributions then one can consider uniqueness with probability one or generic uniqueness.
Generic conditions are often easy to check; they usually just take the form of a bound on the rank as a function of the tensor dimensions. In this paper we derive new, relaxed conditions for the generic uniqueness of CPD and SFS-CPD. 
We resort to the following definitions.
\begin{definition} 
Let $\mu_{\mathbb F}$ be the Lebesgue measure on $\mathbb F^{I\times R}\times \mathbb F^{J\times R}\times \mathbb F^{K\times R}$.
The CPD  of an $I\times J\times K$ tensor of rank $R$ is $generically$ $unique$ if
\begin{equation*}
\mu_{\mathbb F}\{(\mathbf A,\mathbf B,\mathbf C):\ \text{the CPD }  \text{of the tensor }\ [\mathbf A,\mathbf B,\mathbf C]_R
\text{ is not unique }\}=0.
\end{equation*}
\end{definition}
\begin{definition} 
Let $\mu_{\mathbb F}$ be the Lebesgue measure on $\mathbb F^{I\times R}\times \mathbb F^{K\times R}$.
The SFS-CPD  of an $I\times I\times K$ tensor of SFS-rank $R$ is $generically$ $unique$ if
$$
\mu_{\mathbb F}\{(\mathbf A,\mathbf C):\ \text{the SFS-CPD }  \text{of the tensor }\ [\mathbf A,\mathbf A,\mathbf C]_R
\text{ is not unique }\}=0.
$$
\end{definition}
\subsection{Previous results on generic uniqueness of the CPD}
Since the k-rank of a generic  matrix coincides with its minimal dimension, the Kruskal theorem implies the following result: if  
\begin{equation}
R\leq \frac{\min(I,R)+\min(J,R)+\min(K,R)-2}{2},\label{eqkruskalgeneric}
\end{equation}
then the CPD of an $I\times J\times K$ tensor of rank $R$ is generically unique.
Without loss of generality we may assume that $2\leq I\leq J\leq K$. Then \eqref{eqkruskalgeneric} guarantees generic uniqueness for
 $R\leq \min(I+J-2,K)$ and $K< R\leq (I+J+K)/2$. Kruskal's condition is not necessary in general. It was shown in \cite[Proposition 5.2]{ChiantiniandOttaviani}
that if $3\leq I\leq J$ and $\mathbb F=\mathbb C$, then  generic uniqueness holds if 
\begin{equation}
R\leq (I-1)(J-1)\ \text { and }\ (I-1)(J-1)\leq K.\label{eq: Strassen}
\end{equation}
A similar result (involving a different condition in the second part of \eqref{eq: Strassen}) had been obtained before in \cite[Theorem 2.7]{Strassen1983}.
 In the following proposition we collect theoretically proven bounds on $R$ that guarantee generic uniqueness of the CPD for the complimentary case where $K\leq R$.
\begin{proposition} 
\label{propositionalggeom}
Let $2\leq I\leq J\leq K\leq R$.
Then each of the following conditions implies that
 the CPD of an $I\times J\times K$ tensor of rank $R$ is generically unique:
\begin{itemize}
 \item[\textup{(i)}] 
	 $R\leq IJK/(I+J+K-2)-K$, $3\leq I$, $\mathbb F=\mathbb C$ \cite[Corollary 6.2]{Bocci2013}, \cite[Corollary 3.7, $K$ is odd]{Strassen1983};
 \item[\textup{(ii)}]  
 $R\leq 2^{\alpha+\beta-2}$, where  $\alpha$ and $\beta$ are maximal integers such that $2^\alpha\leq I$  and $2^\beta\leq J$
   \cite[Theorem 1.1]{ChiantiniandOttaviani};
 \item[\textup{(iii)}]  $R\leq \frac{I+J+K-2}{2}$ (follows from Kruskal's bound \eqref{eqkruskalgeneric}).
\end{itemize} 
\end{proposition}
The theoretical bounds in Proposition \ref{propositionalggeom} can be further relaxed.
According to the recent paper \cite{Nick2014} the CPD is generically unique (with a few known exceptions) if
 \begin{equation}
 R\leq \left\lceil\frac{IJK}{I+J+K-2}\right\rceil-1, \quad IJK\leq 15000,\label{eq:Nick}
 \end{equation}
 where  $\lceil x \rceil$ denotes the smallest integer not less than $x$.
 The proof of \eqref{eq:Nick} involves the  computation of the kernel of a certain $IJK\times R(I+J+K)$ matrix for a random example with the given dimensions and number of rank-$1$ terms.   
   Similarly, Proposition \ref{prop:generic propositionunstructured} below guarantees generic uniqueness of the CPD
 if at least one of some specially constructed matrices has full column rank.
The conditions in  Proposition \ref{prop:generic propositionunstructured} are formulated in terms of the Khatri-Rao product of $m$-th compound  matrices of $\mathbf A$ and $\mathbf B$.  Recall that {\em the $m$-th compound matrix} of
an
 $I\times R$ matrix $\mathbf A$ (denoted by $\mathcal C_m(\mathbf A)$) is defined for $m\leq\min(I,R)$ and is the 
$\binom{I}{m}\times \binom{R}{m}$  matrix containing the determinants of all $m\times m$
submatrices of $\mathbf A$, arranged with the submatrix index sets in lexicographic order. We refer the reader to
\cite{PartI} for more details on compound matrices.
{\em The Khatri-Rao product} of the matrices $\mathbf A$ and $\mathbf B$ is defined by
$$
\mathbf A\odot \mathbf B=
[\mathbf a_1\ \dots\ \mathbf a_R] \odot
[\mathbf b_1\ \dots\ \mathbf b_R]:=
[\mathbf a_1\otimes\mathbf b_1\ \dots\ \mathbf a_R\otimes\mathbf b_R],
$$
where ``$\otimes$'' denotes the {\em Kronecker product}.

\begin{proposition}\cite[Proposition 1.31]{PartII}\label{prop:generic propositionunstructured}
The  CPD of  an $I\times J\times K$ tensor of rank $R$ is
generically unique if there exist matrices $\mathbf A_0\in \mathbb F^{I\times R}$,  $\mathbf B_0\in \mathbb F^{J\times R}$, and
$\mathbf C_0\in \mathbb F^{K\times R}$ such that at least one of the following conditions holds:
\begin{itemize}
\item[\textup{(i)}]
$\mathcal C_{m_{\mathbf C}}(\mathbf A_0)\odot \mathcal C_{m_{\mathbf C}}(\mathbf B_0)$  has full column rank, where $m_{\mathbf C}=R-\min (K,R)+2$;
\item[\textup{(ii)}]
$\mathcal C_{m_{\mathbf A}}(\mathbf B_0)\odot \mathcal C_{m_{\mathbf A}}(\mathbf C_0)$  has full column rank, where $m_{\mathbf A}=R-\min (I,R)+2$;
\item[\textup{(iii)}]
$\mathcal C_{m_{\mathbf B}}(\mathbf C_0)\odot \mathcal C_{m_{\mathbf B}}(\mathbf A_0)$  has full column rank, where $m_{\mathbf B}=R-\min (J,R)+2$.
\end{itemize}
\end{proposition}

It was shown in \cite{PartI,PartII} that if \eqref{eqkruskalgeneric} holds, then \textup{(i)}--\textup{(iii)} in Proposition \ref{prop:generic propositionunstructured} hold, i.e. Proposition \ref{prop:generic propositionunstructured} is more relaxed than \eqref{eqkruskalgeneric}. To see if \textup{(i)}--\textup{(iii)} hold for given dimensions and rank, it suffices to check a random example
 (more specifically, in which the entries of $\mathbf A_0$, $\mathbf B_0$, $\mathbf C_0$ are drawn from continuous probability
 densities).
\subsection{Previous results on generic uniqueness of the SFS-CPD}
The \\ generic uniqueness of the SFS-CPD has been less studied. 
From Kruskal's condition \eqref{eqkruskal} it follows that if 
\begin{equation}
R\leq \min(I,R)+\frac{\min(K,R)}{2}-1,\label{eqkruskalgenericindscal}
\end{equation}
then the  SFS-CPD of  an $I\times I\times K$ SFS-tensor of SFS-rank $R$ is generically unique. 
To the authors' knowledge the following counterpart of Proposition \ref{prop:generic propositionunstructured}
is the only known result on the generic uniqueness of the SFS-CPD.
\begin{proposition}\cite[Proposition 6.8]{PartII}, \cite[$K=R$]{Psycho2006}\label{prop111overall:indscal}
The  SFS-CPD of  an $I\times I\times K$ SFS-tensor of SFS-rank $R$ is generically unique 
if there exist matrices $\mathbf A_0\in \mathbb F^{I\times R}$ and  $\mathbf C_0\in \mathbb F^{K\times R}$ such that
 $\mathcal C_{m_{\mathbf C}}(\mathbf A_0)\odot \mathcal C_{m_{\mathbf C}}(\mathbf A_0)$ or
$\mathcal C_{m_{\mathbf A}}(\mathbf A_0)\odot \mathcal C_{m_{\mathbf A}}(\mathbf C_0)$
  has full column rank, where $m_{\mathbf C}=R-\min (K,R)+2$ and $m_{\mathbf A}=R-\min (I,R)+2$.
\end{proposition}
\subsection{Contributions of the paper}
In this paper we present new generic uniqueness results for the CPD and SFS-CPD.
Based on deterministic conditions from \cite{PartI,PartII} (namely, Propositions \ref{ProFullUniq1onematrixWm}, \ref{Prop:Indscal1}, and \ref{Prop:Indscal2} further on)
we obtain  theoretically proven bounds on $R$.
\subsubsection{Results on generic uniqueness of the CPD}
The following result  complements the conditions for CPD in Proposition \ref{propositionalggeom}.
\begin{proposition}\label{Propositionunstructured}
 Let  
\begin{align}
2&\leq I\leq J\leq K\leq R,\label{eqmainboundWm0}\\
 R&\leq \frac{I+J+2K-2-\sqrt{(I-J)^2+4K}}{2},\label{eqmainboundWm}
 \end{align}
 or equivalently
 \begin{gather}
 m-1\leq I\leq J \leq K\leq R,\label{eq1newintro}\\
 R\leq (I+1-m)(J+1-m)+m-2,\label{eq3newintro}
 \end{gather}
 where $m=R-K+2$.
 Then the CPD
of an $I\times J\times K$ tensor of rank $R$ is generically unique.
\end{proposition}

Since it is often the case in applications that the largest dimension of a tensor  exceeds its rank,  we explicitly formulate the following special case of Proposition \ref{Propositionunstructured}.  
\begin{corollary}\label{Corollarynew} Let $3\leq I\leq J\leq R\leq K$ and $R\leq (I-1)(J-1)$. Then the CPD of an $I\times J\times K$ tensor of rank $R$ is generically unique.
\end{corollary}

Corollary \ref{Corollarynew} improves the results of \cite[Propositions 5.2]{ChiantiniandOttaviani} and \cite[Theorem 2.7]{Strassen1983}
mentioned above (see \eqref{eq: Strassen}). Namely, the assumption $(I-1)(J-1)\leq K$ in 
\eqref{eq: Strassen} is relaxed as $R\leq K$ and  the statement on generic uniqueness holds both for $\mathbb F=\mathbb C$ and $\mathbb F=\mathbb R$.
It is also interesting to note that for  $\mathbb F=\mathbb C$ the decomposition is generically  not unique  if
$R>(I-1)(J-1)$  \cite[Proposition 2.2]{ChiantiniandOttaviani}. In the case where $\mathbb F = \mathbb C$,  Corollary \ref{Corollarynew} can also be obtained by combining \cite[Propositions 5.2]{ChiantiniandOttaviani} and \cite[Theorem 4.1]{Nick2014}.

Let us compare bound \eqref{eqmainboundWm} with Kruskal's bound \eqref{eqkruskalgeneric}, bound \eqref{eq: Strassen} and bounds from Propositions
\ref{propositionalggeom}--\ref{prop:generic propositionunstructured}. For $I\geq 3$, bound \eqref{eqmainboundWm} improves \eqref{eqkruskalgeneric} by 
$\frac{K-\sqrt{(I-J)^2+4K}}{2}$. For $R=K$, \eqref{eqmainboundWm} coincides with \eqref{eq: Strassen}. Using results of  \cite{PartI,PartII} one can show that
condition \eqref{eqmainboundWm} is more relaxed than Proposition \ref{prop:generic propositionunstructured}.
 In the following examples we present some cases 
where \eqref{eqmainboundWm} is more relaxed than any bound from  Propositions \ref{propositionalggeom}--\ref{prop:generic propositionunstructured} and compare bound \eqref{eqmainboundWm} with the bound in  Proposition \ref{propositionalggeom} \textup{(i)}.
\begin{example} 
By Proposition \ref{propositionalggeom} \textup{(iii)} or Proposition \ref{prop:generic propositionunstructured}, 
the CPD of a generic $4\times 5\times 6$ tensor is unique for $R\leq 6$ and by Proposition
\ref{Propositionunstructured} and \eqref{eq:Nick}, generic  uniqueness is guaranteed for $R\leq 7$ and $R\leq 9$, respectively. 
\end{example}
\begin{example} 
One can easily check that if $K=(I-2)(J-2)$, then the right-hand side of  \eqref{eqmainboundWm} is equal to $R=(I-2)(J-2)+1$.
Thus, by Proposition \ref{Propositionunstructured},
the CPD of an $I\times J\times (I-2)(J-2)$ tensor of rank $(I-2)(J-2)+1$ is generically unique.
In particular, the CPD of an $7\times 8\times 30$ tensor of rank $31$ is generically unique.
It can be shown that this result does not follow from Proposition \ref{propositionalggeom}. 
By \eqref{eq:Nick},  generic uniqueness holds for $R\leq 39$. On the other hand, for
increasing $I$ and $J$, bound \eqref{eq:Nick} becomes harder and harder to verify. For instance, to guarantee
that CPD of an $I\times J\times (I-2)(J-2)$ tensor of rank $(I-2)(J-2)+1$ is generically unique one should compute the kernel of an $IJ(I-2)(J-2)\times (I+J+(I-2)(J-2))((I-2)(J-2)+1)$ matrix, which quickly becomes infeasible  \cite{Nick2014}.
\end{example}
\begin{example} 
We compare the bound in Proposition \ref{propositionalggeom} \textup{(i)} with bound 
\eqref{eqmainboundWm} for $I\times I\times K$ tensors.
By formula manipulation it can easily be shown that if $I=J$ and $\min(9,I)\leq K$, then
the bound in Proposition \ref{propositionalggeom} \textup{(i)}
is more relaxed than  bound \eqref{eqmainboundWm}
if at least  $\frac{5}{2}+\sqrt{2K-\sqrt{K}+\frac{21}{4}}\leq I$
and that 
the bound \eqref{eqmainboundWm} 
is more relaxed than the bound in Proposition \ref{propositionalggeom} \textup{(i)}
if at least  $I\leq 2+\sqrt{2K-\sqrt{K}+3}$.
\end{example}
\subsubsection{Results on generic uniqueness of the SFS-CPD}
If either $R< I$ or $I\leq R< \min(K,2I-2)$ or $\max(I,K)\leq R\leq (2I+K-2)/2$,  then generic uniqueness of the SFS-CPD follows from \eqref{eqkruskalgenericindscal}. Other theoretical bounds on $R$ for the cases $2\leq I<R\leq K$, $2\leq I\leq K\leq R$, and $2\leq K\leq I<R$ are stated in Propositions  \ref{Proposition1.4}, \ref{PropositionINDSCAL}, and \ref{PropositionINDSCAL2}, respectively.
\begin{proposition}\label{Proposition1.4}
Let $4\leq I<R\leq K$ and $R\leq \frac{I^2-I}{2}$.
Then the  SFS-CPD of  an $I\times I\times K$ SFS-tensor of SFS-rank $R$ is generically unique.
\end{proposition}
\begin{proposition}\label{PropositionINDSCAL}
Let
\begin{align}
2&\leq I\leq K\leq R,\nonumber\\ 
 R&\leq \frac{2I+2K+1-\sqrt{8K+8I+1}}{2},\label{eqIndscalIlessK}
 \end{align}
  or equivalently
  \begin{gather}
  	m-1\leq I\leq K\leq R,\nonumber 
  	\\
  	R\leq \frac{I^2+(3-2m)I}{2}+\frac{(m-1)(m-2)}{2}\label{indscal1fromintro2},
  \end{gather}
  where $m=R-K+2$.
 Then the SFS-CPD
of an $I\times I\times K$ tensor of rank $R$ is generically unique.
\end{proposition}

\begin{proposition}\label{PropositionINDSCAL2}
Let 
 \begin{align}
  2&\leq K\leq I\leq R,\label{eqIndscalKlessI1}\\
  R&\leq \frac{K+3I-1-\sqrt{(K-I)^2+2K+6I-3}}{2},\label{eqIndscalKlessI}
  \end{align}
   or equivalently
  \begin{gather}
  m-1\leq K\leq I\leq R,\nonumber\\ 
  R\leq (I+1-m)(K+1-m),\label{eq3new222intro}
  \end{gather}
   where $m=R-I+2$.
  Then the SFS-CPD
of an $I\times I\times K$ tensor of rank $R$ is generically unique.
\end{proposition}

Using results of  \cite{PartI,PartII} one can show that for $\min(I,K)\geq 3$, 
bound \eqref{eqIndscalKlessI} is more relaxed than the bound in
Proposition \ref{prop111overall:indscal}, which is known \cite{PartI,PartII} to be more relaxed than Kruskal's
condition \eqref{eqkruskalgenericindscal}. It can also be shown that for $2\leq I\leq K\leq R$, bound \eqref{eqIndscalIlessK}
is more relaxed than bound in Proposition \ref{prop111overall:indscal} (and hence \eqref{eqkruskalgenericindscal}) in all cases except 
$(I,K)\in\{(2,2), (3,4), (4,4), (5,6), (6,6), (8,8)\}$.
 
\begin{example} 
Kruskal's condition \eqref{eqkruskalgenericindscal} and Proposition \ref{prop111overall:indscal}
guarantee that the SFS-CPD
of an $8\times 8\times 20$ tensor of rank $R$ is generically unique for $R\leq 14$ and $R\leq 20$, respectively \cite[Example 6.14]{PartII}. 
By Proposition \ref{PropositionINDSCAL}, uniqueness  holds also for $R=21$. More generally, if $I\geq 5$, then
by Proposition \ref{PropositionINDSCAL}, the SFS-CPD of an $I\times I\times \frac{I^2-3I}{2}$ tensor of rank $\frac{I^2-3I}{2}+1$ is generically unique.
\end{example}

\subsection{Organization of the paper}
A number of deterministic conditions  for uniqueness of the CPD and SFS-CPD have been obtained in \cite{PartII}.
The main part of the theory in \cite{PartII} was built around conditions that were denoted as \condKm, \condCm, \condUm, and \condWm\ (each succeeding condition is more relaxed than the preceding one, but harder to use). It was shown that 
 condition \eqref{eqkruskalgeneric} and  the conditions in Propositions \ref{prop:generic propositionunstructured}, \ref{prop111overall:indscal}
are generic versions  of the \condKm\ and  \condCm\ based deterministic conditions, respectively. In this paper we  obtain
generic versions of the \condUm\ and  \condWm\ based deterministic conditions from \cite{PartII}. We proceed as follows.
In the first part of Section \ref{Section3} we recall the \condWm\ based deterministic condition for uniqueness of the CPD (Proposition \ref{ProFullUniq1onematrixWm}) and in the first parts of  Sections  \ref{Section4} and \ref{Section5} we derive two \condUm\ based conditions for uniqueness of the SFS-CPD (Propositions \ref{Prop:Indscal1} and \ref{Prop:Indscal2}, respectively).
Then in the second parts of Sections \ref{Section3}, \ref{Section4}, \ref{Section5}  we interpret  the conditions in Propositions 
\ref{Propositionunstructured}, \ref{PropositionINDSCAL}, and \ref{PropositionINDSCAL2} as generic versions of the deterministic Propositions \ref{ProFullUniq1onematrixWm}, \ref{Prop:Indscal1}, and \ref{Prop:Indscal2}, respectively.
Proposition \ref{Proposition1.4} is derived from Proposition  \ref{PropositionINDSCAL} in Section \ref{Sectionformercor}.
Our derivations make use of algebraic geometry. 
Section \ref{Subsection2.2} contains relevant basic  definitions and results.
In Subsection \ref{subsubsection2.2.3}
we summarize these results in a procedure  that may be used in different applications to study generic conditions. 
Although algebraic geometry based approaches have appeared in, e.g., \cite{Bocci2013,ChiantiniandOttaviani,Nick2014, Strassen1983},
the power of the algebraic geometry framework is not yet fully acknowledged in mathematical engineering.
We hope that our paper has some tutorial value in this respect.

\section{Auxiliary results from algebraic geometry}\label{Subsection2.2}
This section is provided to make the paper accessible for readers not familiar with
algebraic geometry. We present a well-known algebraic geometry based method to prove that a set $W\subset\mathbb F^l$
has measure zero, $\mu_{\mathbb F}(W)=0$. For $\mathbb F=\mathbb C$ we introduce the notion of dimension of $W$, $\dim W$, and explain how to compute it.
We show that if $\dim W<l$, then $\mu_{\mathbb C}(W)=0$. The method is summarized and illustrated in Subsections \ref{subsection:zeromeausure} and
\ref{subsection:example}, respectively. In Subsection \ref{subsection:zeromeausure} we also explain that the case $\mathbb F=\mathbb R$ can be reduced to the case $\mathbb F=\mathbb C$.
 In this paper, to prove Proposition \ref{Propositionunstructured} and Propositions
\ref{PropositionINDSCAL}--\ref{PropositionINDSCAL2},
 we will use the method for
\begin{equation*} 
W=\{(\mathbf A,\mathbf B,\mathbf C):\ \text{the CPD of }\ [\mathbf A,\mathbf B,\mathbf C]_R\ \text{ is not unique}\} \subset\mathbb F^{I\times R}\times\mathbb F^{J\times R}\times\mathbb F^{K\times R}
\end{equation*}
($W$ can be considered as a subset of $\mathbb F^l$ with $l=(I+J+K)R$) and
\begin{equation*} 
W=\{(\mathbf A,\mathbf C):\ \text{the SFS-CPD }  \text{of }\ [\mathbf A,\mathbf A,\mathbf C]_R
\text{ is not unique}\} \subset\mathbb F^{I\times R}\times\mathbb F^{K\times R},
\end{equation*}
($W$ can be considered as a subset of $\mathbb F^l$ with $l=(I+K)R$), respectively.
\subsection{Zariski topology} 
A subset $X\subset \mathbb C^n$ is {\em Zariski closed} if there is a set of polynomials $p_1(z_1,\dots,z_n),\dots,p_k(z_1,\dots,z_n)$
such that 
\begin{equation}\label{eq:setX}
X=\{(z_1,\dots,z_n):\ p_1(z_1,\dots,z_n)=0,\dots, p_k(z_1,\dots,z_n)=0\}.
\end{equation}
 A subset $Y\subset \mathbb C^n$ is {\em Zariski open}
if its complement in $\mathbb C^n$ is Zariski closed. A subset $Z \subset\mathbb C^n$ is {\em Zariski locally closed} if it equals  the intersection of an open and a closed subset. {\em The Zariski closure} $\overline W$ of $W\subset\mathbb C^n$  is the smallest closed set such that $W\subset\overline W$. For instance, the set 
$$
Y=\{(z_1,\dots,z_n):\ q_1(z_1,\dots,z_n)\ne 0,\dots, q_l(z_1,\dots,z_n)\ne 0\}.
$$
is Zariski open and the set
\begin{equation*}
\begin{split}
Z=Y\cap X=\{(z_1,\dots,z_n):\ & p_1(z_1,\dots,z_n)=0,\dots, p_k(z_1,\dots,z_n)=0,\\
& q_1(z_1,\dots,z_n)\ne 0,\dots, q_l(z_1,\dots,z_n)\ne 0\}.
\end{split}
\end{equation*}
is Zariski locally closed.
If $W=(0,1)\subset \mathbb C^1$, then the closure of $W$ in the classical Euclidean topology is $[0,1]$ and the closure in the Zariski topology is
the entire $\mathbb C^1$. Indeed, if $\overline{W}=\{z:\ p_1(z)=\dots=p_k(z)=0\}\supset (0,1)$, then $p_i=0$. Hence, $\overline{W}=\mathbb C^1$. 

In the sequel we will consider closed and open subsets only in Zariski topology and, for brevity,  we drop the term ``Zariski''. 
The following lemma follows easily from the above  definitions.
\begin{lemma}\label{Lemma2.1}
\begin{itemize}
\item[\textup{(i)}]
The empty set $\empty$ and the whole space $\mathbb C^n$ are the only subsets of $\mathbb C^n$ that are both open and
closed.
\item[\textup{(ii)}]
Let $Y$ be an open subset of $\mathbb C^n$. Then $\overline Y=\mathbb C^n$.
\end{itemize}
\end{lemma}
\subsection{Dimension of a subset} \label{subsubsecdimset}
With an arbitrary subset $W\subset \mathbb C^n$ one can associate a number, called {\em the dimension of } $W$, $\dim W\in\{0,1,\dots,n\}$  such that $\overline{W}\subsetneq \mathbb C^n$ if and only if $\dim W<n$. In this subsection we give a definition and discuss how the dimension can be computed.

A closed subset $X$ is {\em reducible} if it is the union of two smaller closed subsets $X_1$ and $X_2$, $X= X_1\cup X_2$.
A closed subset $X$ is {\em irreducible} if it is not reducible.
For instance, the subset $X:=\{(z_1,z_2):\ z_1z_2=0\}\subset \mathbb C^2$ is reducible since
$X= X_1\cup X_2$ with $X_i:=\{(z_1,z_2):\ z_i=0\}$; both $X_1$ and $X_2$ are irreducible.

{\em The (topological) dimension} of a subset $W\subset \mathbb C^n$ is the largest  integer $d$ such that there exists a chain
 $X_0\subsetneq X_1\subsetneq\dots\subsetneq X_d\subset \overline{W}$  of distinct irreducible closed subsets of $\overline{W}$. It can be proved that  such $d$ always exists and that $d\leq n$. Since the closure of $\overline{W}$ coincides with $\overline{W}$, it follows immediately from the definition of dimension that $\dim W=\dim\overline{W}$.
 The following properties of the dimension are well-known in algebraic geometry. 
\begin{lemma}\label{lemma:propalggeom}
\begin{itemize}
\item[\textup{(i)}]
Let $W_1\subset W_2\subset\mathbb C^n$, then $\dim W_1\leq\dim W_2\leq n$.
\item[\textup{(ii)}]
Let $W\subset\mathbb C^n$. Then  $\dim W= n$ if and only if $\overline W=\mathbb C^n$.
\item[\textup{(iii)}]
 Let $W\subset W_1\times W_2$ and $\pi_i$ be the projection 
$W_1\times W_2\rightarrow  W_i$, $i=1,2$.  Then 
$$
\max(\dim \pi_1(W),\ \dim \pi_2(W))\leq\ \dim W\leq \dim \pi_1(W)+\dim \pi_2(W).
$$
\item[\textup{(iv)}] Let $W=W_1\cup \dots\cup W_k$. Then $\dim W=\max(\dim W_1,\dots,\dim W_k)$.
\end{itemize}
\end{lemma}
It can be shown that if $W$ is a linear subspace, then $\dim W$ coincides with the well-known definition of  dimension in linear algebra, that is, $\dim W$ is equal to the number of vectors in a basis of $W$. In particular, in the case where $W_1$ and $W_2$ are linear subspaces,
 statements \textup{(i)}--\textup{(iii)} in Lemma \ref{lemma:propalggeom} are well
known in linear algebra.

In the remaining part of this subsection we 
explain a method to obtain a bound on the dimension of a set $W\subset\mathbb C^l$. 
The method is summarized in Procedure \ref{proc}, which, 
together with Lemma \ref{lemma1}, will serve as the main tool for proving generic properties  in this paper.

First, in Subsections \ref{subsubsection2.2.1}--\ref{subsubsection2.2.2}, we address the following auxiliary problem. Given a set $X=\pi(Z)\subset\mathbb C^l$, where the set $Z\subset\mathbb C^n$ is of the special form \eqref{eq:setZ} and $\pi$ is the projection $\pi:\mathbb C^n\rightarrow\mathbb C^l$, we want to obtain a bound on $\dim X$. 

\subsubsection{Construction of  set $Z$ and determination of $\dim Z$}\label{subsubsection2.2.1}
 Let $p_1,\dots,$ $p_{n-m}$, $q_1,\dots,q_{n-m}$ be polynomials in the variables $z_1,\dots,z_m$. We define the open subset
$$
Y=\{(z_1,\dots,z_m):\ q_1(z_1,\dots,z_m)\ne 0,\dots,  q_{n-m}(z_1,\dots,z_m)\ne 0\}\subset \mathbb C^m.
$$
Then, by Lemma \ref{Lemma2.1} \textup{(ii)}, $\overline{Y}=\mathbb C^m$. Hence, by Lemma \ref{lemma:propalggeom} \textup{(ii)}, $\dim Y=m$. By definition, the set $Z\subset\mathbb C^n$ is the image of $Y$ under the mapping
\begin{equation*} 
\phi:\ (z_1,\dots,z_m)\in Y\mapsto\left(z_1,\dots,z_m, \frac{p_1(z_1,\dots,z_m)}{q_1(z_1,\dots,z_m)},\dots,\frac{p_{n-m}(z_1,\dots,z_m)}{q_{n-m}(z_1,\dots,z_m)}\right)\in Z,
\end{equation*}
that is, $Z=\phi(Y)$. It is clear that the projection of $Z$  onto the first $m$ coordinates of $\mathbb C^n$  coincides with $Y$.
Hence, by Lemma \ref{lemma:propalggeom} \textup{(iii)}, $\dim Z\geq m$. The following lemma is well known, it states
that $\dim Z= m$. In other words, the dimension of the image $\phi(Y)$  cannot exceed the dimension of $Y$. Note that  Lemma
\ref{lemma:2.8} requires  $p_1,\dots,p_{n-m}$, $q_1,\dots,q_{n-m}$  to be polynomials in the variables $z_1,\dots,z_m$.
\begin{lemma}\label{lemma:2.8}
Let $p_1,\dots,p_{n-m},q_1,\dots,q_{n-m}$ be polynomials in the variables\\ $z_1,\dots,z_m$ and
\begin{equation}\label{eq:setZ}
\begin{split}
Z=\{&(z_1,\dots,z_m,z_{m+1},\dots,z_n):\\
&z_{m+1}:=\frac{p_1(z_1,\dots,z_m)}{q_1(z_1,\dots,z_m)},\dots, z_n:=\frac{p_{n-m}(z_1,\dots,z_m)}{q_{n-m}(z_1,\dots,z_m)},\\
&q_1(z_1,\dots,z_m)\ne 0,\dots,  q_{n-m}(z_1,\dots,z_m)\ne 0,\ z_1,\dots,z_m\in\mathbb C\}\subset \mathbb C^n.
\end{split}
\end{equation}
Then $\overline Z$ is irreducible and $\dim \overline{Z}=m$.
\end{lemma}

In this paper we call the variables $z_1,\dots,z_m$ and $z_{m+1},\dots,z_n$ in \eqref{eq:setZ}  {\em ``independent'' parameters} and {\em ``dependent'' parameters}, respectively. Thus, Lemma \ref{lemma:2.8} formalizes the fact that $\dim Z$ coincides with the number of its ``independent parameters''.


\subsubsection{Construction of  projection $\pi$ and bound on $\dim \pi(Z)$}\label{subsubsection2.2.2}
Let $\pi$ be the projection 
\begin{equation}\label{eq:projection}
\pi:\ \mathbb C^n\rightarrow \mathbb C^l,\quad
\pi(z_1,\dots,z_n)=(z_{k+1},\dots,z_m,\dots,z_{k+l}),
\end{equation}
for certain $k$, $l$ such that  $k+1\leq m\leq k+l\leq n$. We  assume additionally that
$1\leq k$ and $l\leq m$. We consider the set $\pi(Z)$. Thus,
 $\pi$ drops at least one of the ``independent'' parameters  $z_1,\dots,z_m$ but not all of them; $\pi$ may also drop ``dependent'' parameters
 $z_{m+1},\dots,z_n$, even all of them. 
By Lemma \ref{lemma:propalggeom} \textup{(i)}, $\dim \pi(Z)\leq \dim\mathbb C^l=l$. Now  we explain that this trivial bound
may be further improved so that we obtain $\dim \pi(Z)< l$.

Let $f$ denote  the restriction of $\pi$ to $Z$,
\begin{equation}\label{eq:mapf}
f:Z\rightarrow \mathbb C^l,\quad f(z_1,\dots,z_n)=(z_{k+1},\dots,z_{k+l})\ \text{for all }\ (z_1,\dots,z_n)\in Z,
\end{equation}
yielding that $\pi(Z)=f(Z)$.
Denote by $f^{-1}(s_{k+1},\dots,s_{k+l})\subset Z$  the preimage of the point $(s_{k+1},\dots,s_{k+l})\in f(Z)$:
\begin{equation*} 
\begin{split}
f^{-1}(s_{k+1},\dots,s_{k+l})= &\ \{(z_1,\dots,z_n)\in Z: f(z_1,\dots,z_n)=(s_{k+1},\dots,s_{k+l})\}\\
=&\ \{(z_1,\dots,z_n)\in Z:\ z_{k+1}=s_{k+1},\dots,z_{k+l}=s_{k+l}\}.
\end{split}
\end{equation*}
The following lemma easily follows from the ``Fiber dimension Theorem'' \cite[Theorem 3.7, p. 78]{Perrin2008}, it relates the dimension of the preimage with the dimension of the projection.

\begin{lemma}\label{Theorem:FDTh} Let $Z$ and $f$ be defined by \eqref{eq:setZ} and \eqref{eq:mapf}, respectively.
Suppose that 
\begin{equation}\label{eq:fpreim}
\dim f^{-1}(s_{k+1},\dots,s_{k+l})\geq d\ \text{ for all } (s_{k+1},\dots,s_{k+l})\in f(Z).
\end{equation}
Then 
 $ (\dim \overline{\pi(Z)}=)\dim \overline{f(Z)}\leq m-d$.
\end{lemma}

Thus, to obtain the  bound
$\dim \pi(Z)<l$, it suffices  to show that $d>m-l$.

The results of Subsections \ref{subsubsection2.2.1}--\ref{subsubsection2.2.2} are summarized in the following procedure.
\begin{procedure}\label{procaux} Input: a subset $X=\pi(Z)\subset\mathbb C^l$, where 
$Z\subset\mathbb C^n$ ($n\geq l$) is of the form \eqref{eq:setZ} and  $\pi$ is of the form \eqref{eq:projection}.\\
Output: bound on $\dim X$.
\begin{itemize}
\item[(i)]
 Set $m:=\dim Z$ (by Lemma \ref{lemma:2.8}).
\item[(ii)]
 Find $d$ such that  \eqref{eq:fpreim} holds.
\item[(iii)]
$\dim X\leq m-d$ (by Lemma \ref{Theorem:FDTh}).
\end{itemize}
\end{procedure}
\subsubsection{A method to obtain a bound on $\dim W\subset\mathbb C^l$}\label{subsubsection2.2.3}
In this subsection we consider the
following problem: given a set of points $W\subset\mathbb C^l$ that satisfy a certain property, we want to 
show that $\dim W<l$.

First we ``parameterize'' the problem: we find a larger subset $\widehat{Z}\subset\mathbb C^n$ and a projection $\pi:\mathbb C^n\rightarrow\mathbb C^l$ such that $W=\pi(\widehat{Z})$. Our parameterizations are such that the set $\widehat{Z}$ is included into  a finite union of subsets $Z_u\subset\mathbb C^n$, $\widehat{Z}\subset\bigcup Z_u$, and that all $Z_u$ are of the form \eqref{eq:setZ}. (For example,  if $W$ is the set of $2\times 2$ matrices with zero eigenvalue, then $W=\pi(\widehat{Z})$, where $\widehat{Z}=
\{(\mathbf A,\mathbf f): \mathbf A\mathbf f=\mathbf 0,\ \mathbf f\ne \mathbf 0\}\subset \mathbb C^{2\times 2}\times \mathbb C^2
$
and $\pi:\mathbb C^{2\times 2}\times \mathbb C^2\rightarrow\mathbb C^{2\times 2}$. Obviously,
$\widehat{Z}:=Z_1\cup Z_2$, where
$
Z_i=\{(\mathbf A,\mathbf f): \mathbf A\mathbf f=\mathbf 0,\ f_i\ne 0\}$, and it can be verified that, indeed, $Z_1$ and $Z_2$ are of the form
\eqref{eq:setZ}.)
Since $W=\pi(\widehat{Z})$ and $\widehat{Z}\subset\bigcup Z_u$, from
Lemma \ref{lemma:propalggeom} \textup{(i)}, \textup{(iv)} it follows that
\begin{equation*}
\dim W=\dim\pi(\widehat Z)\leq \dim\pi (\bigcup Z_u)= \dim\bigcup \pi(Z_u)=
\max\limits_u\dim \pi(Z_u)\ 
\end{equation*}
or $\dim W\leq\max\limits_u\dim X_u$, where $X_u=\pi(Z_u)$.
Thus, to obtain a bound on $\dim W$ one should obtain bounds on $\dim X_u$ for all $u$. This can be done
by following the steps in Procedure \ref{procaux} for $X=X_u$. 

 The results of Subsection \ref{subsubsecdimset} are summarized in the following procedure. 
  \begin{procedure}\label{proc} Input: a set of points $W\subset\mathbb C^l$ that satisfy a certain property.\\
  Output: bound on $\dim W$.\\  
{\bf Phase I: ``Parameterization''.}
 \begin{itemize}
 \item[(1)] Express $W$ as $\pi(\widehat{Z})$: $W=\pi(\widehat{Z})$,  where  $\widehat{Z}\subset \mathbb C^n$ and
  $\pi$ is of the form \eqref{eq:projection}.
  \item[(2)] Express $\widehat{Z}$ as part of a finite union $\widehat{Z}\subset\bigcup Z_u$, where all $Z_u$ are of the form \eqref{eq:setZ}.
 \end{itemize}
 {\bf Phase II: Obtaining a bound on $\dim W$.}
 \begin{itemize}
 \item[(3)] For all values of $u$: apply Procedure \ref{procaux} for $X=X_u=\pi(Z_u)$, obtain a bound on $\dim X_u$,
 $\dim X_u\leq l_u$.
 \item[(4)]   $\dim W\leq\max\limits_u l_u$.
 \end{itemize}
 \end{procedure}
In none of the cases in this paper,  the values $l_u$ in step (3) of Procedure \ref{proc} will depend on $u$.
 Thus, we will apply Procedure \ref{procaux} only once, e.g. for $u=1$.

\subsection{Zariski closed proper subsets have measure zero}\label{subsection:zeromeausure}
The following \\ lemma is well known. 
We include a proof since we do not know an explicit reference where such a proof can be found.
\begin{lemma}\label{lemma1}
Let $W\subset \mathbb C^l$, $\dim W<l$ and  $W_{\mathbb R} := W \cap \mathbb R^l$.
Then $\mu_{\mathbb C}\{W\}=0$ and $\mu_{\mathbb R}\{W_{\mathbb R}\}=0$, where
$\mu_{\mathbb C}$ and $\mu_{\mathbb R}$ denote the Lebesgue measures on $\mathbb C^l$ and $\mathbb R^l$, respectively.
\end{lemma}
\begin{proof}
We may assume that $\overline{W}$ is defined by \eqref{eq:setX}. Then 
 $\overline{W}$ is the zero set of the polynomials $p_1,\dots, p_k$. The results follow from the well-known fact
that the zero set of a nonzero polynomial  has measure zero both on $\mathbb C^l$ and $\mathbb R^l$.
\end{proof}

As our overall strategy for showing that
a subset $W\subset \mathbb  F^l$ has measure zero, we will use Procedure \ref{proc} and Lemma \ref{lemma1} as follows. If $\mathbb F=\mathbb C$, then  we follow the steps in Procedure \ref{proc} to show that $\dim W< l$, and conclude, by Lemma \ref{lemma1}, that $\mu_{\mathbb C}(W)=0$. If
$\mathbb F=\mathbb R$, then first, we  extend  $W\subset\mathbb R^l$ to a subset $W_{\mathbb C}\subset\mathbb C^l$ by letting all parameters in $W$
take values in $\mathbb C$; second, we follow the steps in Procedure \ref{proc} to show that $\dim W_{\mathbb C}< l$, and conclude, by Lemma \ref{lemma1}, that   $\mu_\mathbb R(W)= \mu_\mathbb R(W_{\mathbb C} \cap \mathbb R^l)=0$.
\subsection{Example}\label{subsection:example}
 To illustrate our approach we prove the well-known fact that 
two  generic square matrices of the same size do not share eigenvalues.
\begin{example} 
Let $W=\{(\mathbf A,\mathbf B):\ \mathbf A \text{ and } \mathbf B\ \text{have a common eigenvalue}\}$
be a subset of $\mathbb C^{n\times n}\times\mathbb C^{n\times n}$. We claim that $\mu_{\mathbb C}(W)=0$, where $\mu_{\mathbb C}$  is the Lebesgue measure on $\mathbb C^{n\times n}\times \mathbb C^{n\times n}$. By Lemma \ref{lemma1} it is sufficient to prove that $\dim W\leq 2n^2-1$.
To obtain a bound on $\dim W$ we follow the steps in Procedure \ref{proc}.

{\bf Phase I: ``Parameterization''.}

\textup{(1)} It is clear that
 $(\mathbf A,\mathbf B)\in W$ if and only if
there exist $\lambda\in\mathbb C$  and nonzero vectors $\mathbf f$ and $\mathbf g$ such that $\mathbf A\mathbf f=\lambda\mathbf f$ and $\mathbf B\mathbf g=\lambda\mathbf g$. Hence, $W=\pi(\widehat Z)$, where
$$
\widehat Z=\{(\mathbf A,\mathbf B,\lambda,\mathbf f,\mathbf g):\ \mathbf A\mathbf f=\lambda\mathbf f,\ \mathbf B\mathbf g=\lambda\mathbf g,\ \mathbf f\ne\mathbf 0,\ \mathbf g\ne\mathbf 0\}
$$
is a subset of $\mathbb C^{n\times n}\times \mathbb C^{n\times n}\times \mathbb C\times\mathbb C^n\times\mathbb C^n$ 
and $\pi$ is the projection onto the first two factors
$$
\pi:\ \mathbb C^{n\times n}\times \mathbb C^{n\times n}\times \mathbb C\times\mathbb C^n\times\mathbb C^n\rightarrow 
\mathbb C^{n\times n}\times \mathbb C^{n\times n}.
$$
\qquad\textup{(2)} We 
represent $\widehat Z$ as a finite union of  sets of the form \eqref{eq:setZ}: 
$$
\widehat Z=\bigcup\limits_{1\leq u,v\leq n} Z_{u,v},\quad Z_{u,v}:=\{(\mathbf A,\mathbf B,\lambda,\mathbf f,\mathbf g)
:\ \mathbf A\mathbf f=\lambda\mathbf f,\ \mathbf B\mathbf g=\lambda\mathbf g,\ \ f_u\ne0,\ g_v\ne 0\}.
$$ 
We show that all $Z_{u,v}$ are of the form \eqref{eq:setZ}.
 To simplify the presentation we restrict ourselves to the case $u=1$ and $v=1$. The general case can be proved in the same way. Since, by assumption, $f_1\ne 0$ and $g_1\ne 0$, we can express $\mathbf a_1$ and $\mathbf b_1$ via
 $\lambda$, $\mathbf a_2,\dots,\mathbf a_n$, $\mathbf b_2,\dots,\mathbf b_n$, $\mathbf f$, and $\mathbf g$. Hence,
\begin{equation*}
\begin{split}
Z_{1,1}:=&\ \{(\mathbf A,\mathbf B,\lambda,\mathbf f,\mathbf g)\in Z,\ f_1\ne0,\ g_1\ne 0\}\\
=&\ \{(\mathbf a_1 = (\lambda\mathbf f-\mathbf a_2f_2-\dots-\mathbf a_nf_n)/f_1,\mathbf a_2,\dots,\mathbf a_n,\\
           &\ \ \  \mathbf b_1 = (\lambda\mathbf f-\mathbf b_2g_2-\dots-\mathbf b_ng_n)/g_1,\mathbf b_2,\dots,\mathbf b_n,\lambda,\mathbf f,\mathbf g):\ f_1\ne0,\ g_1\ne 0\}
           \end{split}
\end{equation*}
is indeed of the form \eqref{eq:setZ}, where $z_1,\dots,z_m$ correspond to the value  $\lambda$ and the entries of $\mathbf a_2,\dots,\mathbf a_n$, $\mathbf b_2,\dots,\mathbf b_n$, $\mathbf f$, $\mathbf g$ and where $z_{m+1},\dots,z_n$ correspond to the
 entries of $\mathbf a_1$ and  $\mathbf b_1$.

{\bf Phase II: Obtaining a bound on $\dim W$.}

\begin{itemize}
\item[\textup{(3)}] To obtain bounds on $\dim \pi(Z_{u,v})$ we follow the steps in Procedure \ref{procaux}
 for $X=\pi(Z_{u,v})$. W.l.o.g. we again restrict ourselves to the case 
$u=1$ and $v=1$.
\item[\textup{(3i)}] By Lemma \ref{lemma:2.8}, $\dim Z_{1,1}=1+n(n-1)+n(n-1)+n+n=2n^2+1$. 
\item[\textup{(3ii)}] Let $f:Z_{1,1}\rightarrow \mathbb C^{n\times n}\times  \mathbb C^{n\times n}$ denote  the restriction of $\pi$ to $Z_{1,1}$: 
  $$
  f(\mathbf A,\mathbf B,\lambda,\mathbf f,\mathbf g)=(\mathbf A, \mathbf B),\quad  
  (\mathbf A,\mathbf B,\lambda,\mathbf f,\mathbf g)\in Z_{1,1}.
  $$
 From the definition of $Z_{1,1}$ it follows that if 
$(\mathbf A,\mathbf B,\lambda,\mathbf f,\mathbf g)\in Z_{1,1}$,
   then \\ $(\mathbf A,\mathbf B,\lambda,\alpha\mathbf f,\beta\mathbf g)\in Z_{1,1}$,
    where $\alpha$ and $\beta$ are arbitrary nonzero values. (Indeed, if $\mathbf A\mathbf f=\lambda\mathbf f$ and $\mathbf B\mathbf g=\lambda\mathbf g$, then $\mathbf A(\alpha \mathbf f)=\lambda(\alpha \mathbf f)$ and $\mathbf B(\beta \mathbf g)=\lambda(\beta \mathbf g)$.)
 Hence,
 $$
 f^{-1}(\mathbf A, \mathbf B)\supset\{(\mathbf A,\mathbf B,\lambda,\alpha\mathbf f,\beta\mathbf g):\ \alpha\ne 0,\ \beta\ne 0\}.
  $$
  By Lemma \ref{lemma:propalggeom} \textup{(iii)},
  $
    \dim f^{-1}(\mathbf A, \mathbf B)\geq \dim \{(\alpha \mathbf f, \beta \mathbf g):\ \alpha\ne 0,\ \beta\ne 0\} \geq\\ \dim \{(\alpha f_1, \beta g_1):\ \alpha\ne 0,\ \beta\ne 0\}
    $. Since $\dim \{(\alpha f_1, \beta g_1):\ \alpha\ne 0,\ \beta\ne 0\}=\dim\mathbb C^2=2$, it follows that $\dim f^{-1}(\mathbf A, \mathbf B)\geq 2=:d$.
\item[\textup{(3iii)}] By Lemma \ref{Theorem:FDTh}, $\dim \overline{\pi(Z_{1,1})}\leq\dim Z_{1,1}-d\leq 2n^2+1-2=2n^2-1$. Note that precisely the property that $\mathbf A$ and $\mathbf B$ share an eigenvalue has allowed us to find a projection that reduces the dimension. What remains is a little effort to show that $d=2$ implies that having eigenvalues in common is a subgeneric property.
\end{itemize}
\qquad\textup{(4)} Hence, $\dim W\leq\max\limits_{u,v}\dim \pi(Z_{u,v}) =\dim \pi(Z_{1,1})=2n^2-1$.

Since $W\subset \mathbb C^{n\times n}\times \mathbb C^{n\times n}$ and $\dim W\leq 2n^2-1$, 
	it follows from  Lemma \ref{lemma1} that $\mu_{\mathbb C}(W)=0$.
\end{example}

\section{Uniqueness of the CPD and proof of Proposition \ref{Propositionunstructured}}\label{Section3}

In the sequel, $\omega(\lambda_1, \dots,\lambda_R)$ denotes the number of nonzero entries of $[\lambda_1\ \dots\ \lambda_R]^T$.
The following  condition \condWm\  was introduced in \cite{PartI, PartII} in terms of $m$-th compound matrices.
In this paper we will use the following (equivalent)  definition.
\begin{definition}\label{Def:Wm}
We say that  condition \condWm\  holds for the triplet of matrices $(\mathbf A,\mathbf B,\mathbf C)\in \mathbb F^{I\times R}\times \mathbb F^{J\times R}
\times \mathbb F^{K\times R}$ if $\omega(\lambda_1, \dots,\lambda_R)\leq m-1$ whenever
 \begin{gather*}
r_{\mathbf A\Diag{\lambda_1,\dots,\lambda_R}\mathbf B^T}\leq m-1\quad  \text{ for }\quad [\lambda_1\ \dots\ \lambda_R]^T\in\textup{range}(\mathbf C^T). 
\end{gather*}
\end{definition}
Since the rank of the product $\mathbf A\Diag{\lambda_1,\dots,\lambda_R}\mathbf B^T$ does not exceed the rank of any of the factors and since
$r_{\Diag{\lambda_1,\dots,\lambda_R}}=\omega(\lambda_1,\dots,\lambda_R)$, we have the implication
\begin{equation}\label{eqdirectimplication}
\omega(\lambda_1,\dots,\lambda_R)\leq m-1\quad\Rightarrow\quad r_{\mathbf A\Diag{\lambda_1,\dots,\lambda_R}\mathbf B^T}\leq m-1.
\end{equation}
Condition  \condWm\  in Definition \ref{Def:Wm} means that the opposite of the implication in \eqref{eqdirectimplication} holds
for all $[\lambda_1\ \dots\ \lambda_R]\in\textup{range}(\mathbf C^T)\subset\mathbb C^R$. 
We now give a set of deterministic conditions, among which a \condWm-type condition, that guarantee CPD uniqueness. These conditions will be checked for a generic tensor in the proof of Proposition \ref{Propositionunstructured}.
\begin{proposition}(see \cite[Proposition 1.22]{PartII})\label{ProFullUniq1onematrixWm}
Let $\mathcal T=[\mathbf A, \mathbf B, \mathbf C]_R$ and
$m_{\mathbf C}:=R-r_{\mathbf C}+2$.
Assume that
\begin{itemize}
\item[\textup{(i)}]
$\max(\min( k_{\mathbf A}, k_{\mathbf B}-1),\ \min( k_{\mathbf A}-1, k_{\mathbf B}))+k_{\mathbf C}\geq R+1$;
\item[\textup{(ii)}]
condition \condWmC\ holds for the triplet $(\mathbf A,\mathbf B,\mathbf C)$;
\item[\textup{(iii)}]
$\mathbf A\odot\mathbf B$ has full column rank.
\end{itemize}
Then $r_{\mathcal T}=R$ and the CPD of tensor $\mathcal T$  is unique.
\end{proposition}

{\em Proof of Proposition \ref{Propositionunstructured}.}
We show that
\begin{equation}\label{eq:fullkrank1}
\mu_{\mathbb F}\{(\mathbf A,\mathbf B,\mathbf C):\ \text{\textup{(i)} or \textup{(ii)} or \textup{(iii)} of Proposition \ref{ProFullUniq1onematrixWm} does not hold}\}=0.
\end{equation}
Since, under condition \eqref{eqmainboundWm0},
\begin{equation*} 
\mu_{\mathbb F}\{(\mathbf A,\mathbf B,\mathbf C):\ k_{\mathbf A}<I \text{ or } k_{\mathbf B}<J \text{ or }  k_{\mathbf C}<K\}=0,
\end{equation*}
it follows that \eqref{eq:fullkrank1} holds if and only if
\begin{equation*} 
\begin{split}
\mu_{\mathbb F}\{(\mathbf A,\mathbf B,\mathbf C):\ &\text{\textup{(i)} or \textup{(ii)} or \textup{(iii)} of Proposition \ref{ProFullUniq1onematrixWm} does not hold and}\\
&k_{\mathbf A}=I,\ k_{\mathbf B}=J,\ k_{\mathbf C}=K \}=0.
\end{split}
\end{equation*}

{\bf Condition \textup{(i):}}
By \eqref{eqmainboundWm0}--\eqref{eqmainboundWm},
$$
R+1\leq I+K +\frac{J-I-\sqrt{(I-J)^2+4K}}{2}<I+K\leq J+K,
$$
which easily implies that condition \textup{(i)} of Proposition \ref{ProFullUniq1onematrixWm}
holds for all matrices $\mathbf A$, $\mathbf B$, and $\mathbf C$ such that 
$k_{\mathbf A}=I$,\ $k_{\mathbf B}=J$,\ $k_{\mathbf C}=K$.

{\bf Condition \textup{(iii)}:}
 By \eqref{eq3newintro}, \eqref{eq1newintro},
$$
R\leq IJ-1-(m-1)(I+J-m)<IJ-1.
$$
It is well-known \cite[Theorem 3]{AlmostSure2001} that
$$
\mu_1\{(\mathbf A,\mathbf B):\ \text{\textup{(iii)} of Proposition \ref{ProFullUniq1onematrixWm} does not hold}\}=0
$$
if and only if $R\leq IJ$, where $\mu_1$ denotes the Lebesgue measure on $\mathbb F^{I\times R}\times \mathbb F^{J\times R}$.
Fubini's theorem \cite[Theorem C, p. 148]{halmos1974measure} allows us to extend this to a statement for $(\mathbf A,\mathbf B,\mathbf C)$:
$$
\mu_{\mathbb F}\{(\mathbf A,\mathbf B,\mathbf C):\ \text{\textup{(iii)} of Proposition \ref{ProFullUniq1onematrixWm} does not hold} \}=0.
$$

{\bf Condition \textup{(ii)}:}
Let
\begin{equation} \label{eq:fullkrank27878}
\begin{split}
W=\{(\mathbf A,\mathbf B,\mathbf C):\ &\text{\textup{(ii)} of Proposition \ref{ProFullUniq1onematrixWm} does not hold and}\\
&k_{\mathbf A}=I,\ k_{\mathbf B}=J,\ k_{\mathbf C}=K \}\subset \mathbb F^{I\times R}\times \mathbb F^{J\times R}\times \mathbb F^{K\times R}.
\end{split}
\end{equation}
To complete the proof of  Proposition \ref{Propositionunstructured} we need to show that 
$\mu_{\mathbb F}\{W\}=0$.  By Lemma \ref{lemma1}, it is sufficient to prove that, for $\mathbb F=\mathbb C$, the closure of $W$
is not the entire space $\mathbb C^{I\times R}\times \mathbb C^{J\times R}\times \mathbb C^{K\times R}$, which is equivalent
(see discussion in Subsection \ref{subsection:zeromeausure}) to
\begin{equation} \label{eq: dimofX}
\dim \overline{W}\leq IR+JR+KR-1.
\end{equation}
To prove bound \eqref{eq: dimofX} we follow the steps in Procedure \ref{proc}.

{\bf Phase I: ``Parameterization''.}

\textup{(1)} We associate $W$ with a certain $\pi(\widehat Z)$.
By Definition \ref{Def:Wm},  condition \condWm\  does not hold for the triplet $(\mathbf A,\mathbf B,\mathbf C)$ if and only if
there exist values  $\lambda_1,\dots,\lambda_R$, and  matrices $\widetilde{\mathbf A}\in \mathbb C^{I\times (m-1)}$, $\widetilde{\mathbf B}\in \mathbb C^{J\times (m-1)}$ such that
$$
 \mathbf A\Diag{\lambda_1,\dots,\lambda_R}\mathbf B^T=\widetilde{\mathbf A}\widetilde{\mathbf B}^T, \ \
 [\lambda_1\ \dots\ \lambda_R]^T\in\textup{Range}(\mathbf C^T),\ \
 $$
but $\omega(\lambda_1,\dots,\lambda_R)\geq m$.
We claim that if additionally
$k_{\mathbf A}=I$ and $k_{\mathbf B}=J$, then
$\omega(\lambda_1,\dots,\lambda_R)\geq I$. Indeed, if
$m\leq \omega(\lambda_1,\dots,\lambda_R) <I\leq J$, then  by the Frobenius inequality,
\begin{equation*}
\begin{split}
m-1\geq&\ r_{\widetilde{\mathbf A}\widetilde{\mathbf B}^T}=
r_{\mathbf A\Diag{\lambda_1,\dots,\lambda_R}\mathbf B^T}\\
\geq&\ r_{\mathbf A\Diag{\lambda_1,\dots,\lambda_R}}
+r_{\Diag{\lambda_1,\dots,\lambda_R}\mathbf B^T}-
r_{\Diag{\lambda_1,\dots,\lambda_R}}\\
=&\ \omega(\lambda_1,\dots,\lambda_R)+\omega(\lambda_1,\dots,\lambda_R)-\omega(\lambda_1,\dots,\lambda_R)=
\omega(\lambda_1,\dots,\lambda_R)\geq m,
\end{split}
\end{equation*}
which is a contradiction.
Hence, $W$ in \eqref{eq:fullkrank27878} can be expressed as
\begin{align}
W=\{(\mathbf A,\mathbf B,\mathbf C):\ \text{\condWm\  does }& \text{not hold for the triplet } (\mathbf A,\mathbf B,\mathbf C), \nonumber
\\
&\qquad\qquad\qquad\qquad\qquad\qquad k_{\mathbf A}=I,\ k_{\mathbf B}=J,\ k_{\mathbf C}=K \}\nonumber\\
=\{(\mathbf A,\mathbf B,\mathbf C):\ \text{there exist }& \lambda_1,\dots,\lambda_R\in\mathbb C,\ \widetilde{\mathbf A}\in \mathbb C^{I\times (m-1)},\ \text{ and }\ \widetilde{\mathbf B}\in \mathbb C^{J\times (m-1)}\ \nonumber \\
\text{ such that }\ & \mathbf A\Diag{\lambda_1,\dots,\lambda_R}\mathbf B^T=\widetilde{\mathbf A}\widetilde{\mathbf B}^T,\label{eq:ABC1}\\
&[\lambda_1\ \dots\ \lambda_R]^T\in\textup{Range}(\mathbf C^T), \label{eq:ABC2}\\
& k_{\mathbf A}=I,\ k_{\mathbf B}=J,\ k_{\mathbf C}=K, \label{eq:ABC3}\\
&\omega(\lambda_1,\dots,\lambda_R)\geq I\}.\nonumber
\end{align}
It is now clear that $W=\pi(\widehat Z)$, where 
$$
\widehat Z=\{(\mathbf A, \mathbf B, \mathbf C,
 \lambda_1,\dots,\lambda_R,
 \widetilde{\mathbf A}, \widetilde{\mathbf B}):\ \text{ \eqref{eq:ABC1}--\eqref{eq:ABC3} hold}\}
$$
is a subset of
$ 
\mathbb C^{I\times R}\times \mathbb C^{J\times R}\times \mathbb C^{K\times R}\times\mathbb C^R\times \mathbb C^{I\times (m-1)}\times \mathbb C^{J\times (m-1)}
$ 
and $\pi$ is the projection onto the first three factors 
\begin{equation*} 
\begin{split}
\pi:\ \mathbb C^{I\times R}\times \mathbb C^{J\times R}\times \mathbb C^{K\times R}\times\mathbb C^R\times \mathbb C^{I\times (m-1)}\times &\mathbb C^{J\times (m-1)}\rightarrow\\
 &\mathbb C^{I\times R}\times  \mathbb C^{J\times R}\times  \mathbb C^{K\times R}.
 \end{split}
\end{equation*}
\textup{(2)} Since
\begin{equation}\label{eq:newlambdaeq}
\omega(\lambda_1,\dots,\lambda_R)\geq I\ \ \Leftrightarrow\ \ 
\lambda_{u_1}\cdots\lambda_{u_I}\ne 0\ \text{ for some }\ 1\leq u_1<\dots<u_I\leq R 
\end{equation}
we obtain
\begin{equation*}
\widehat Z= \bigcup\limits_{1\leq u_1<\dots<u_I\leq R}
\{
(\mathbf A, \mathbf B, \mathbf C,
 \lambda_1,\dots,\lambda_R,
 \widetilde{\mathbf A}, \widetilde{\mathbf B}):\ \text{\eqref{eq:ABC1}--\eqref{eq:ABC3} hold, } \lambda_{u_1}\cdots\lambda_{u_I}\ne 0\}.
\end{equation*}
Let $\mathbf A_{u_1,\dots,u_I}$ denote the submatrix of $\mathbf A$ formed by columns $u_1,\dots,u_I$.
	Since \eqref{eq:ABC3} is  more restrictive  than the conditions $\det\mathbf A_{u_1,\dots,u_I}\ne 0$ and $k_{\mathbf C}=K$, it follows that
$\widehat Z\subset\bigcup\limits_{1\leq u_1<\dots<u_I\leq R} Z_{u_1,\dots,u_I}$, where
\begin{equation*}
\begin{split}
Z_{u_1,\dots,u_I}:=\{
(\mathbf A, \mathbf B, \mathbf C,
 \lambda_1,\dots,\lambda_R,
 \widetilde{\mathbf A}, \widetilde{\mathbf B}):\ &\text{ \eqref{eq:ABC1}--\eqref{eq:ABC2} hold, } \lambda_{u_1}\cdots\lambda_{u_I}\ne 0,\\
 \det\mathbf A_{u_1,\dots,u_I}\ne 0,\ k_{\mathbf C}=K\}.
 \end{split}
\end{equation*}
We show that all $Z_{u_1,\dots,u_I}$ are of the form \eqref{eq:setZ}.
 To simplify the presentation we restrict ourselves to the case $(u_1,\dots, u_I)=(1,\dots,I)$. The general case can be proved in the same way.
 Let the matrices $\mathbf A$, $\mathbf B$ and $\mathbf C$ be partitioned as
 $$
 \mathbf A=[\underbracket{\ \bar{\mathbf A}\ }_{I}\ \underbracket{\ \bar{\bar{\mathbf A}}\ }_{R-I}],\qquad
 \mathbf B=[\underbracket{\ \bar{\mathbf B}\ }_{I}\ \underbracket{\ \bar{\bar{\mathbf B}}\ }_{R-I}],\qquad
 \mathbf C=[\underbracket{\ \bar{\mathbf C}\ }_{K}\ \underbracket{\ \bar{\bar{\mathbf C}}\ }_{R-K}],
 $$
 so that $\bar{\mathbf A}=\mathbf A_{1,\dots,I}$.
    By \eqref{eq:ABC1},
  \begin{equation*} 
   \bar{\mathbf B}=\left[\left(\bar{\mathbf A}\Diag{\lambda_{1},\dots,\lambda_I}\right)^{-1}\left( \widetilde{\mathbf A}\widetilde{\mathbf B}^T - 
  \bar{\bar{\mathbf A}}\Diag{\lambda_{I+1},\dots,\lambda_R}\bar{\bar{\mathbf B}}^T\right)\right]^T.
    \end{equation*}
    By \eqref{eq:ABC2},  there exists $\mathbf x\in\mathbb C^K$ such that 
   $\matr{\lambda_1}{\lambda_R}^T=\mathbf C^T\mathbf x$ or, equivalently,  $\matr{\lambda_1}{\lambda_K}^T= \bar{\mathbf C}^T\mathbf x$ and $\matr{\lambda_{K+1}}{\lambda_R}^T= \bar{\bar{\mathbf C}}^T\mathbf x$.  Hence,
   \begin{equation*} 
    \matr{\lambda_{K+1}}{\lambda_R}^T=\bar{\bar{\mathbf C}}^T\bar{\mathbf C}^{-T}\matr{\lambda_1}{\lambda_K}^T.
       \end{equation*}
      In other words, by Cramer's rule, each entry of $\bar{\mathbf B}$  and each of the values $\lambda_{K+1},\dots,\lambda_R$  can be written as
     a ratio of two polynomials in  entries of $\mathbf A$, $\bar{\bar{\mathbf B}}$, $\mathbf C$,  $\widetilde{\mathbf A}$, $\widetilde{\mathbf B}$ and the values  $\lambda_1,\dots,\lambda_K$. By the assumptions
      $\lambda_1\cdots\lambda_I\ne 0$, $\det\bar{\mathbf A}\ne 0$, and $k_{\mathbf C}=K$ (yielding that $\det\bar{\mathbf C}\ne 0$ ), the denominator polynomial is nonzero.
   
    Hence, $Z_{1,\dots,I}$
 is indeed of the form \eqref{eq:setZ}, where  $z_1,\dots,z_m$ correspond to the entries of
 $\mathbf A$, $\bar{\bar{\mathbf B}}$, $\mathbf C$,  $\widetilde{\mathbf A}$, $\widetilde{\mathbf B}$ and the values  $\lambda_1,\dots,\lambda_K$, and where $z_{m+1},\dots,z_n$  correspond to the entries of $\bar{\mathbf B}$ and the values $\lambda_{K+1},\dots,\lambda_R$. \\
 {\bf Phase II: Obtaining a bound on $\dim W$.}
 
 \textup{(3)}  To obtain bounds on $\dim \pi(Z_{u_1,\dots,u_I})$ we follow the steps in Procedure \ref{procaux} for $X=\pi(Z_{u_1,\dots,u_I})$ W.l.o.g.  we again restrict ourselves to the case $(u_1,\dots,u_I)=(1.\dots,I)$.
 
 \textup{(3i)} By Lemma \ref{lemma:2.8},
$ \dim Z_{1,\dots,I}=IR+J(R-I)+KR+K+I(m-1)+J(m-1)$.

\textup{(3ii)} Let $f:Z_{1,\dots,I}\rightarrow \mathbb C^{I\times R}\times  \mathbb C^{J\times R}\times  \mathbb C^{K\times R}$ denote  the restriction of $\pi$ to $Z_{1,\dots,I}$: 
   $$
   f(\mathbf A, \mathbf B, \mathbf C,
   \lambda_1,\dots,\lambda_R, \widetilde{\mathbf A}, \widetilde{\mathbf B})=(\mathbf A, \mathbf B, \mathbf C),\quad  
   (\mathbf A, \mathbf B, \mathbf C,
     \lambda_1,\dots,\lambda_R, \widetilde{\mathbf A}, \widetilde{\mathbf B})\in Z_{1,\dots,I}.
   $$
  From the definition of $Z_{1,\dots,I}$ it follows that if 
 $(\mathbf A, \mathbf B, \mathbf C,
  \lambda_1,\dots,\lambda_R, \widetilde{\mathbf A}, \widetilde{\mathbf B})\in Z_{1,\dots,I}$,
    then
    $(\mathbf A, \mathbf B, \mathbf C,
       \alpha\lambda_1,\dots,\alpha\lambda_R,
       \widehat{\mathbf A}\mathbf T, \alpha\widehat{\mathbf B}\mathbf T^{-T})\in Z_{1,\dots,I}$,
     where $\alpha$ is an arbitrary nonzero value,  $\widehat{\mathbf A}$ is an arbitrary full column rank matrix  such that $\textup{range}(\widehat{\mathbf A})\supset \textup{range}(\widetilde{\mathbf A})$,
     $\mathbf T$ is an arbitrary nonsingular   $(m-1)\times (m-1)$ matrix, and $\widehat{\mathbf B}$ satisfies  $\widehat{\mathbf A}\widehat{\mathbf B}^T=\widetilde{\mathbf A}\widetilde{\mathbf B}^T$.
  Hence,
  $$
  f^{-1}(\mathbf A, \mathbf B, \mathbf C)\supset\{\mathbf A, \mathbf B, \mathbf C,
   \alpha\lambda_1,\dots,\alpha\lambda_R,
   \widehat{\mathbf A}\mathbf T, \alpha\widehat{\mathbf B}\mathbf T^{-T}:\ \alpha\ne 0,\ \det\mathbf T\ne 0\}.
   $$
   By Lemma \ref{lemma:propalggeom} \textup{(iii)},
   $
     \dim f^{-1}(\mathbf A, \mathbf B, \mathbf C)\geq \dim \{\alpha\lambda_1, \widehat{\mathbf A}\mathbf T:\ \alpha\ne 0,\ \det\mathbf T\ne 0\}
     $. Since, by construction, the matrix $\widehat{\mathbf A}$ has full column rank and, by assumption, $\lambda_1\ne 0$ and $m-1\leq I$, it follows that
   $
   \dim \{\alpha\lambda_1, \widehat{\mathbf A}\mathbf T:\ \alpha\ne 0,\ \det\mathbf T\ne 0\}=1+(m-1)^2=:d.
   $
   Thus, $\dim f^{-1}(\mathbf A, \mathbf B, \mathbf C)\geq d$.
   
   \textup{(3iii)} By Lemma \ref{Theorem:FDTh} and 
   \eqref{eq3newintro}, 
   \begin{equation*}
   \begin{split}
   \dim \overline{\pi(Z_{1,\dots,I})}\leq&\ \dim Z_{1,\dots,I}-d\\
   =&\ IR+J(R-I)+KR+K+I(m-1)+J(m-1)-1-(m-1)^2\\
   \leq&\ IR+JR+KR-1.
   \end{split}
   \end{equation*}
\qquad \textup{(4)} Hence, $\dim W\leq\max\limits_{1\leq u_1<\dots<u_I\leq R}\dim \pi(Z_{u_1,\dots,u_I})\leq IR+JR+KR-1$. 

Since $W\subset \mathbb C^{I\times R}\times \mathbb C^{J\times R}\times \mathbb C^{K\times R}$ and 
$\dim W\leq IR+JR+KR-1$, it follows from  Lemma \ref{lemma1} that $\mu_{\mathbb C}(W)=0$.
\qquad\endproof
\section{Uniqueness of the SFS-CPD and proof of Proposition \ref{PropositionINDSCAL}}\label{Section4}
The following  condition \condUm\  was introduced in \cite{PartI, PartII} in terms of $m$-th compound matrices.
In this paper we will use the following (equivalent)  definition.
\begin{definition}\label{Def:Um}
We say that  condition \condUm\  holds for the pair $(\mathbf A,\mathbf B)\in\mathbb F^{I\times R}\times \mathbb F^{J\times R}$ if
$\omega(\lambda_1,\dots,\lambda_R)\leq m-1$ whenever $r_{\mathbf A\Diag{\lambda_1,\dots,\lambda_R}\mathbf B^T}\leq m-1$.
\end{definition}

Condition  \condUm\  in Definition \ref{Def:Um} means that the opposite of the implication in \eqref{eqdirectimplication} holds
for all $[\lambda_1\ \dots\ \lambda_R]\in\mathbb C^R$. 
Thus, if the matrix $\mathbf C$ has full column rank (implying $\textup{range}(\mathbf C^T)=\mathbb C^R$), then \condWm\ (see Definition \ref{Def:Wm}) is equivalent to \condUm. Otherwise,  \condWm\ is more relaxed than \condUm:
 \begin{flalign}
  &\ \text{\condUm\ holds for the pair }(\mathbf A,\mathbf B)\quad\Rightarrow\quad \text{\condWm\ holds for the triplet } (\mathbf A,\mathbf B,\mathbf C).&\label{eq:Umimplies1}
  \end{flalign}
  The following implication was  proved in  \cite[Lemmas 3.2, 3.7]{PartI}:
  \begin{flalign}
   &\ \text{\condUm\ holds for the pair }(\mathbf A,\mathbf B)\quad\Rightarrow\quad \mathbf A\odot\mathbf B \text{ has full column rank.}&
  \label{eq:Umimplies2}
  \end{flalign}
The following deterministic  condition for  uniqueness of the SFS-CPD will be checked for a generic SFS-tensor in the
proof of Proposition \ref{PropositionINDSCAL}.
\begin{proposition}\label{Prop:Indscal1}
Let $\mathcal T=[\mathbf A, \mathbf A, \mathbf C]_R$ and
$m_{\mathbf C}:=R-r_{\mathbf C}+2$. Assume that
\begin{itemize}
\item[\textup{(i)}] $k_{\mathbf A}+k_{\mathbf C}\geq R+2$;
\item[\textup{(ii)}] condition \condUmC\ holds for the pair $(\mathbf A,\mathbf A)$.
\end{itemize}
Then $r_{\mathcal T}=R$ and the SFS-CPD of tensor $\mathcal T$  is unique.
\end{proposition}
\begin{proof}
We show that \textup{(i)}--\textup{(iii)} in Proposition \ref{ProFullUniq1onematrixWm} hold for
$\mathcal T=[\mathbf A, \mathbf B, \mathbf C]_R$ with $\mathbf B=\mathbf A$: \textup{(i)} is obvious from $k_{\mathbf B}=k_{\mathbf A}$, and 
\textup{(ii)} and \textup{(iii)} follow from \eqref{eq:Umimplies1} and \eqref{eq:Umimplies2}, respectively.
\end{proof}

{\em Proof of  Proposition \ref{PropositionINDSCAL}.} We show that
\begin{equation}\label{eq:fullkrank1indscal1}
\mu_{\mathbb F}\{(\mathbf A,\mathbf C):\ \text{\textup{(i)} or \textup{(ii)} of Proposition \ref{Prop:Indscal1} does not hold}\}=0.
\end{equation}
Since 
\begin{equation*} 
\mu_{\mathbb F}\{(\mathbf A,\mathbf C):\ \mathbf A \text{ has a zero minor or }  k_{\mathbf C}<K\}=0,
\end{equation*}
it follows that \eqref{eq:fullkrank1indscal1} holds if and only if
\begin{equation*} 
\begin{split}
\mu_{\mathbb F}\{(\mathbf A,\mathbf C):\ &\text{\textup{(i)} or \textup{(ii)}  of Proposition \ref{Prop:Indscal1}  does not hold,}\\
&\text{all minors of }\mathbf A\ \text{are nonzero, and}\ k_{\mathbf C}=K \}=0.
\end{split}
\end{equation*}
It is clear that if all minors of $\mathbf A$ are nonzero, then $k_{\mathbf A}=I$.

{\bf Condition \textup{(i)}:} By \eqref{eqIndscalIlessK},
$$
R+2\leq I+K+\max\left(\frac{-K+2}{2},\frac{5-\sqrt{8K+8I+1}}{2}\right)
\leq I+K,
$$
which means that condition \textup{(i)} of Proposition \ref{Prop:Indscal1}
holds for all matrices $\mathbf A$ and $\mathbf C$ such that $k_{\mathbf A}=I$ and $k_{\mathbf C}=K$.

{\bf Condition \textup{(ii)}:}
Let
\begin{equation} \label{eq:fullkrank2indscal}
\begin{split}
W=\{(\mathbf A,\mathbf C):\ &\text{\textup{(ii)} of Proposition \ref{Prop:Indscal1} does not hold,}\\
&\text{all minors of }\mathbf A\ \text{are nonzero, and}\ k_{\mathbf C}=K \}\subset \mathbb F^{I\times R}\times \mathbb F^{K\times R}.
\end{split}
\end{equation}
To complete the proof of  Proposition \ref{PropositionINDSCAL} we need to show that 
$\mu_{\mathbb F}\{W\}=0$.  By Lemma \ref{lemma1}, it is sufficient to prove that for $\mathbb F=\mathbb C$, the closure of $W$ is not the entire space $\mathbb C^{I\times R}\times\mathbb C^{K\times R}$, which is equivalent
(see discussion in Subsection \ref{subsection:zeromeausure}) to
\begin{equation} 
\dim \overline{W}\leq IR+KR-1.\label{eq:boundindscal1}
\end{equation}
To prove bound \eqref{eq:boundindscal1} we follow the steps in Procedure \ref{proc}.

{\bf Phase I: ``Parameterization''.}

\textup{(1)} We associate $W$ with a certain $\pi(\widehat Z)$.
By Definition \ref{Def:Um},  condition \condUm\  does not hold for the pair $(\mathbf A,\mathbf A)$ if and only if
there exist values  $\lambda_1,\dots,\lambda_R$, and a matrix $\widetilde{\mathbf A}\in \mathbb C^{I\times (m-1)}$ such that
$$
 \mathbf A\Diag{\lambda_1,\dots,\lambda_R}\mathbf A^T=\widetilde{\mathbf A}\widetilde{\mathbf A}^T, \ \
  $$
but $\omega(\lambda_1,\dots,\lambda_R)\geq m$.
We claim that if additionally
$k_{\mathbf A}=I$, then
$\omega(\lambda_1,\dots,\lambda_R)\geq I$.
Indeed, if
$m=m_{\mathbf C}\leq \omega(\lambda_1,\dots,\lambda_R) <I$, then  by the Frobenius inequality,
\begin{equation*}
\begin{split}
m-1\geq&\  r_{\widetilde{\mathbf A}\widetilde{\mathbf A}^T}
=r_{\mathbf A\Diag{\lambda_1,\dots,\lambda_R}\mathbf A^T}\\
\geq&\  r_{\mathbf A\Diag{\lambda_1,\dots,\lambda_R}}+
r_{\Diag{\lambda_1,\dots,\lambda_R}\mathbf A^T}-
r_{\Diag{\lambda_1,\dots,\lambda_R}}\\
=&\ \omega(\lambda_1,\dots,\lambda_R)+\omega(\lambda_1,\dots,\lambda_R)-\omega(\lambda_1,\dots,\lambda_R)=
\omega(\lambda_1,\dots,\lambda_R)\geq m,
\end{split}
\end{equation*}
which is a contradiction.
Hence, $W$ in \eqref{eq:fullkrank2indscal} can be expressed as
\begin{align}
W=\{(\mathbf A,\mathbf C):&\ \text{\condUm\  does not hold for the pair } (\mathbf A,\mathbf A), \nonumber\\
& \text{all minors of }\mathbf A\ \text{are nonzero, and}\ k_{\mathbf C}=K \} \nonumber\\
=\{(\mathbf A,\mathbf C):&\ \text{there exist } \lambda_1,\dots,\lambda_R\in\mathbb C,\ \widetilde{\mathbf A}\in \mathbb C^{I\times (m-1)} \text{ such that } \nonumber \\
\ & \mathbf A\Diag{\lambda_1,\dots,\lambda_R}\mathbf A^T=\widetilde{\mathbf A}\widetilde{\mathbf A}^T,\ \label{eq:ABC2indscal1}\\
&\text{all minors of }\mathbf A\ \text{are nonzero,}\label{eq:ABC2indscal2}\\
& k_{\mathbf C}=K,\ \text{and} \label{eq:ABC2indscal3}\\
&\omega(\lambda_1,\dots,\lambda_R)\geq I\}.\label{eq:ABC2indscal4} 
\end{align}
Obviously, $W=\pi^*(\widehat Z^*)$, where 
$
\widehat Z^*=\{(\mathbf A,\mathbf C, \lambda_1,\dots,\lambda_R,
 \widetilde{\mathbf A}):\ \text{ \eqref{eq:ABC2indscal1}--\eqref{eq:ABC2indscal4} hold}\}
 \subset
 \mathbb C^{I\times R}\times \mathbb C^{K\times R}\times\mathbb C^R\times \mathbb C^{I\times (m-1)}
$
and $\pi^*$ is the projection onto the first two factors, 
$$
\pi^*:\  \mathbb C^{I\times R}\times \mathbb C^{K\times R}\times\mathbb C^R\times \mathbb C^{I\times (m-1)}\rightarrow
  \mathbb C^{I\times R}\times \mathbb C^{K\times R}.
$$
In step (2) of  Procedure \ref{proc} we should express $\widehat Z^*$  as part of a finite union of  sets of the form \eqref{eq:setZ}. That is,
some entries of $\mathbf A$, $\mathbf C$, $\tilde{\mathbf A}$ and some of the values $\lambda_1,\dots,\lambda_R$ (corresponding to ``dependent'' parameters) must be rational functions of the remaining entries of $\mathbf A$, $\mathbf C$, $\tilde{\mathbf A}$ and of the remaining  values of $\lambda_1,\dots,\lambda_R$ (corresponding to ``independent'' parameters).
 Since \eqref{eq:ABC2indscal1} is nonlinear in the entries of $\mathbf A$ and $\tilde{\mathbf A}$, it is not easy, if possible at all, 
 to find such rational functions.  To solve this problem, we further parametrize $W$, namely, we parametrize an $I\times I$ submatrix of $\mathbf A$ by its LDU decomposition, that is, we extend $\widehat Z^*$ to a  ``larger'' $\widehat Z$ by
 adding the new variables $\mathbf L$, $\mathbf D$, and $\mathbf U$. In step (2) we will show that 
 the new, LDU-based, parametrization   will turn identity \eqref{eq:ABC2indscal1} into a UDU decomposition. The properties of the UDU decomposition will imply that the entries of $\mathbf U$ and $\mathbf D$ will be rational functions of the remaining parameters involved in \eqref{eq:ABC2indscal1}.
 
By definition, set
\begin{equation}\label{eq:indscal2hatZ}
\begin{split}
\widehat Z&:= \bigcup\limits_{1\leq u_1<\dots<u_I\leq R} 
\{
(\mathbf A, \mathbf C,\mathbf L,\mathbf D,\mathbf U,
 \lambda_1,\dots,\lambda_R,
 \widetilde{\mathbf A}):\\
  &\qquad\ \ \qquad\ \ \qquad\ \ \text{ \eqref{eq:ABC2indscal1}--\eqref{eq:ABC2indscal3} hold, } \lambda_{u_1}\cdots\lambda_{u_I}\ne 0,\\
  &\qquad\ \ \qquad\ \ \qquad\ \ \mathbf A_{u_1,\dots,u_I}=\mathbf L\mathbf D\mathbf U,\quad \mathbf D\ \text{ is nonsingular and diagonal},\\
  &\qquad\ \ \qquad\ \ \qquad\ \ \mathbf L\ \text{ is unit lower triangular,}\  \mathbf U\ \text{ is unit upper triangular}\}\\
  &\subset
  \mathbb C^{I\times R}\times\mathbb C^{K\times R}\times\mathbb C^{I\times I}\times\mathbb C^{I\times I}\times \mathbb C^{I\times I}\times\mathbb C^R\times \mathbb C^{I\times (m-1)}.
  \end{split}
  \end{equation} 
 where $\mathbf A_{u_1,\dots,u_I}$ denotes the submatrix of $\mathbf A$ formed by columns $u_1,\dots,u_I$.
Let also $\pi$ be the projection onto the first two factors
\begin{equation*} 
\pi:\  \mathbb C^{I\times R}\times\mathbb C^{K\times R}\times\mathbb C^{I\times I}\times\mathbb C^{I\times I}\times \mathbb C^{I\times I}\times\mathbb C^R\times \mathbb C^{I\times (m-1)}\rightarrow
\mathbb C^{I\times R}\times\mathbb C^{K\times R}.
\end{equation*}
We show that $W=\pi(\widehat Z)$. 
Since, by \eqref{eq:ABC2indscal2}, all principal minors in the upper left-hand corner of
 $\mathbf A_{u_1,\dots,u_I}$ are nonzero,
 it follows from  properties of the LDU decomposition  that
 there exist $I\times I$ matrices
 $\mathbf L$, $\mathbf D$, and $\mathbf U$ such that  $\mathbf L$ is unit lower triangular, $\mathbf U$ is unit upper triangular,
 $\mathbf D$ is nonsingular and  diagonal, and $\mathbf A_{u_1,\dots,u_I}=\mathbf L\mathbf D\mathbf U$. In other words,
 the set $\pi(\widehat Z)$ remains the same if the equality constraint $\mathbf A_{u_1,\dots,u_I}=\mathbf L\mathbf D\mathbf U$ 
 in all subsets in the union in \eqref{eq:indscal2hatZ} is dropped. Hence, by \eqref{eq:newlambdaeq},
 \begin{equation*}
 \begin{split}
 \pi(\widehat Z) &= \pi\big( \{
 (\mathbf A, \mathbf C,\mathbf L,\mathbf D,\mathbf U,
  \lambda_1,\dots,\lambda_R,
  \widetilde{\mathbf A}):\\
   &\qquad\ \ \qquad\ \ \text{ \eqref{eq:ABC2indscal1}--\eqref{eq:ABC2indscal4} hold, }\ \mathbf D\ \text{ is nonsingular and diagonal,}\\
   &\qquad\ \ \qquad\ \ \mathbf L\ \text{ is unit lower triangular,}\  \mathbf U\ \text{ is unit upper triangular}\}
 \big)=W.
 \end{split}
 \end{equation*}
  \textup{(2)}
Dropping condition \eqref{eq:ABC2indscal2} in all subsets in the union in \eqref{eq:indscal2hatZ}, we obtain that
   $\widehat Z\subset \bigcup\limits_{1\leq u_1<\dots<u_I\leq R} Z_{u_1,\dots,u_I}$, where
 \begin{equation*}
 \begin{split}
 Z_{u_1,\dots,u_I}:=\{
 (\mathbf A, \mathbf C,\mathbf L,\mathbf D,&\mathbf U,
  \lambda_1,\dots,\lambda_R,
  \widetilde{\mathbf A}):\
   \text{ \eqref{eq:ABC2indscal1} and \eqref{eq:ABC2indscal3} hold, } \lambda_{u_1}\cdots\lambda_{u_I}\ne 0,\\
   &\ \mathbf A_{u_1,\dots,u_I}=\mathbf L\mathbf D\mathbf U,\quad \mathbf D\ \text{ is nonsingular and diagonal},\\
   &\ \mathbf L\ \text{ is unit lower triangular,}\  \mathbf U\ \text{ is unit upper triangular}\}.
   \end{split}
 \end{equation*}
We show that all $Z_{u_1,\dots,u_I}$ are of the form \eqref{eq:setZ}.
To simplify the presentation we restrict ourselves to the case $(u_1,\dots, u_I)=(1,\dots,I)$. 
The general case can be proved in the same way.

For brevity  we set $\bar{\mathbf A}=\mathbf A_{1,\dots,I}$ and  $\bar{\bar{\mathbf A}}=\mathbf A_{I+1,\dots,R}$, so that
  $  \mathbf A=[\bar{\mathbf A}\ \bar{\bar{\mathbf A}}]$ with
$\bar{\mathbf A}=\mathbf L\mathbf D\mathbf U$. We show that $Z_{1,\dots, I}$ is of the form \eqref{eq:setZ},
where $z_1,\dots,z_m$  correspond to  the entries of  $\bar{\bar{\mathbf A}}$, $\mathbf C$, $\mathbf L$, $\mathbf D$, $\widetilde{\mathbf A}$ 
and the values $\lambda_{I+1},\dots,\lambda_R$, where $z_{m+1},\dots,z_n$ correspond to 
 the entries of $\bar{\mathbf A}$, $\mathbf U$ and  the values $\lambda_1,\dots,\lambda_I$, and where the inequality
  constraint $\lambda_1\cdots\lambda_I\ne 0\nonumber$ can be replaced by an equivalent inequality constraint which depends only on
   the entries of  $\bar{\bar{\mathbf A}}$, $\mathbf C$, $\mathbf L$, $\mathbf D$, $\widetilde{\mathbf A}$ 
  and the values $\lambda_{I+1},\dots,\lambda_R$.

   Substituting $\mathbf A=[\mathbf L\mathbf D\mathbf U\ \bar{\bar{\mathbf A}}]$ into 
 \eqref{eq:ABC2indscal1} we obtain
 $$
 \left(\mathbf L\mathbf D\mathbf U\right)
 \Diag{\lambda_1,\dots,\lambda_I}\left(\mathbf L\mathbf D\mathbf U\right)^T
 +\bar{\bar{\mathbf A}}
 \Diag{\lambda_{I+1},\dots,\lambda_R}
 \bar{\bar{\mathbf A}}^T=\widetilde{\mathbf A}\widetilde{\mathbf A}^T
 $$
 or
 \begin{align}
 \label{eq:4.4fg}
 &\mathbf U\Diag{\lambda_1,\dots,\lambda_I}\mathbf U^T=\mathbf M,\ \text{where}\\
 &\mathbf M:=(\mathbf L\mathbf D)^{-1}\left(\widetilde{\mathbf A}\widetilde{\mathbf A}^T-\bar{\bar{\mathbf A}}\Diag{\lambda_{I+1},\dots,\lambda_R}
 \bar{\bar{\mathbf A}}^T\right)(\mathbf L\mathbf D)^{-T}.\label{eq:4.4fgr}
 \end{align}
 It follows from \eqref{eq:4.4fg}, the assumption $\lambda_1\cdots\lambda_I\ne 0$, and properties of the UDU decomposition
 that each of the values $\lambda_1,\dots ,\lambda_I$ and each entry of $\mathbf U$  can be written as
 the ratio of two polynomials in  entries of $\mathbf M$, that the
  denominator polynomial can always be chosen equal to one of the principal minors in the lower right hand corner of $\mathbf M$,
  and that
  \begin{align*}
 &\lambda_1\cdots\lambda_I\ne 0\nonumber\\
  \Leftrightarrow\
 &\text{all principal minors in the lower right hand corner of } \mathbf M \text{ are nonzero}\nonumber\\ 
  \begin{split}
   \xLeftrightarrow{\eqref{eq:4.4fgr}}\ &\text{all principal minors in the lower right hand corner of the matrix }\\
    &(\mathbf L\mathbf D)^{-1}\left(\widetilde{\mathbf A}\widetilde{\mathbf A}^T-\bar{\bar{\mathbf A}}\Diag{\lambda_{I+1},\dots,\lambda_R}  \bar{\bar{\mathbf A}}^T\right)(\mathbf L\mathbf D)^{-T} \text{ are nonzero.} 
    \end{split}
 \end{align*}
 Hence, $Z_{1,\dots,I}$ is indeed of the form \eqref{eq:setZ}.   \\
 {\bf Phase II: Obtaining a bound on $\dim W$.}

\textup{(3)}  To obtain bounds on $\dim \pi(Z_{u_1,\dots,u_I})$ we follow the steps in Procedure \ref{procaux} for $X=\pi(Z_{u_1,\dots,u_I})$. W.l.o.g.  we again restrict ourselves to the case $(u_1,\dots,u_I)=(1.\dots,I)$.
 
 \textup{(3i)} By Lemma \ref{lemma:2.8}, $\dim \overline{Z}_{1,\dots,I}$ is equal to the 
 sum of  the number of entries of
  $\bar{\bar{\mathbf A}}$, $\mathbf C$, $\mathbf L$, $\mathbf D$, $\widetilde{\mathbf A}$, and $[\lambda_{I+1},\dots,\lambda_R]$,
 \begin{align*}
 \dim \overline{Z}_{1,\dots,I}=&\ 
 I(R-I)+  KR+  (I^2-I)/2+I+ (m-1)I+ (R-I)\\
 =&\ IR+KR+R-\frac{I^2-I(2m-3)}{2}.
  \end{align*}
 \qquad\textup{(3ii)} Let $f:Z_{1,\dots,I}\rightarrow \mathbb C^{I\times R}\times  \mathbb C^{K\times R}$ denote  the restriction of $\pi$ to $Z_{1,\dots,I}$: 
  $$
   f(\mathbf A, \mathbf C, \mathbf L,\mathbf D,\mathbf U, \lambda_1,\dots,\lambda_R, \widetilde{\mathbf A})=(\mathbf A,\mathbf C),\quad  
   (\mathbf A,\mathbf C, \mathbf L,\mathbf D,\mathbf U,\lambda_1,\dots,\lambda_R, \widetilde{\mathbf A})\in Z_{1,\dots,I}.
   $$
   We need to find $d$ such that $\dim f^{-1}(\mathbf A,\mathbf C)\geq d$ for all $(\mathbf A,\mathbf C)\in f(Z_{1,\dots,I})$.
    
We make the assumption:
    \begin{equation}\label{eq:assumption}
     \begin{split}
       \text{if }\ &(\mathbf A,\mathbf C,\mathbf L,\mathbf D,\mathbf U, \lambda_1,\dots,\lambda_R, \widetilde{\mathbf A})\in Z_{1,\dots,I},\\ \text{ then }\ 
         &(\mathbf A,\mathbf C,\mathbf L,\mathbf D,\mathbf U,\widehat{\lambda}_1,\dots,\widehat{\lambda}_R, \widehat{\mathbf A})\in Z_{1,\dots,I}
          \end{split}
    \end{equation}
   for some   $\widehat{\lambda}_1,\dots,\widehat{\lambda}_R$  and full column rank matrix $\widehat{\mathbf A}$. That assumption
   \eqref{eq:assumption} indeed holds will be demonstrated later. 
  Since
   $$
   \mathbf A\Diag{\alpha^2\widehat{\lambda}_1,\dots,\alpha^2\widehat{\lambda}_R}\mathbf A^T=
   \alpha^2\mathbf A\Diag{\widehat{\lambda}_1,\dots,\widehat{\lambda}_R}\mathbf A^T=
   \alpha^2\widehat{\mathbf A}\widehat{\mathbf A}^T=(\alpha \widehat{\mathbf A}\mathbf T)
   (\alpha \widehat{\mathbf A}\mathbf T)^T
   $$
   for any nonzero $\alpha$ and an $(m-1)\times (m-1)$ matrix $\mathbf T$ such that $\mathbf T\mathbf T^T=\mathbf I$,
   it follows from the definition $Z_{1,\dots,I}$ that  $(\mathbf A,\mathbf C,\mathbf L,\mathbf D,\mathbf U, \alpha^2\widehat{\lambda}_1,\dots,\alpha^2\widehat{\lambda}_R,  \alpha\widehat{\mathbf A}\mathbf T)\in Z_{1,\dots,I}$.  
     Thus,
   \begin{equation*}
   \begin{split}
    f^{-1}(\mathbf A,\mathbf C)\supset\{
    \mathbf A,\mathbf C,\mathbf L,\mathbf D,\mathbf U, \alpha^2&\widehat{\lambda}_1,\dots,\alpha^2\widehat{\lambda}_R,
       \alpha\widehat{\mathbf A}\mathbf T:\\ 
       &\alpha\in\mathbb C, \ \mathbf T\in\mathbb C^{(m-1)\times(m-1)}\ \alpha\ne 0,\ 
           \mathbf T\mathbf T^T=\mathbf I\}.
           \end{split}
     \end{equation*}
      Hence, by Lemma \ref{lemma:propalggeom} \textup{(iii)} and by Lemma \ref{lemma:moregencomport} below,
      \begin{equation*}
      \begin{split}
     \dim f^{-1}(\mathbf A,\mathbf C)\geq&\  \dim \{\alpha\widehat{\mathbf A}\mathbf T:\ \alpha\in\mathbb C, \ \mathbf T\in\mathbb C^{(m-1)\times(m-1)},\ \alpha\ne 0,\ 
              \mathbf T\mathbf T^T=\mathbf I\}\\               =&\ \frac{(m-1)(m-2)}{2}+1:=d.
              \end{split}
     \end{equation*}
     Let us now prove assumption \eqref{eq:assumption}. If $r_{\widetilde{\mathbf A}}=m-1$, then we just set $\widehat{\mathbf A}:=\widetilde{\mathbf A}$ and $\widehat{\lambda}_i:=\lambda_i$.  If  $r_{\widetilde{\mathbf A}}=k\leq m-2$, then we construct $\widehat{\lambda}_1,\dots,\widehat{\lambda}_R$ and $\widehat{\mathbf A}$ from
        $\lambda_1,\dots,\lambda_R$, some (generic) values $\mu_1,\dots,\mu_{m-1-k}$, the matrix $\widetilde{\mathbf A}$, and certain $m-1-k$ columns of   $\mathbf A_{1,\dots,I}$ as follows. Let $\mathbf S$ be a full column rank $I\times k$ matrix such that $\mathbf S\mathbf S^T=
        \widetilde{\mathbf A}\widetilde{\mathbf A}^T$.
        Since the columns of   $\mathbf A_{1,\dots,I}$ form a basis of $\mathbb C^I$,
        the columns of $\mathbf S$ can be extended to a basis of $\mathbb C^I$ by adding suitable $m-1-k$
        columns of $\mathbf A_{1,\dots,I}$. In particular,  the $I\times (m-1)$ matrix
        $\widehat{\mathbf A}:=[\mathbf S\ \mathbf A_{i_1,\dots,i_{m-1-k}}\Diag{\mu_1,\dots,\mu_{m-1-k}}]$
        has full column rank for certain indicies $i_1,\dots,i_{m-1-k}$ such that $1\leq i_1<\dots<i_{m-1-k}\leq I$ and for any nonzero values $\mu_1,\dots,\mu_{m-1-k}$. 
        We construct $\widehat{\lambda}_1,\dots,\widehat{\lambda}_R$ from
       $ \lambda_1,\dots,\lambda_R$
        as
       \begin{flalign*}
             \qquad\qquad\qquad\qquad&\widehat{\lambda}_{i_j}=\lambda_{i_j}+\mu_j^2,\ &\text{for } j\in\{1,\dots,m-1-k\},& &\qquad\qquad\\
       &\widehat{\lambda}_i =   \lambda_i,\ &\text{for } i\not\in\{i_1,\dots,i_{m-1-k}\}.\quad & &\qquad\qquad\qquad\qquad
              \end{flalign*}
       We choose nonzero $\mu_1,\dots,\mu_{m-1-k}$ such that
      $\widehat{\lambda}_1\cdots,\widehat{\lambda}_I\ne 0$.
        From \eqref{eq:ABC2indscal1} and the identity $\mathbf S\mathbf S^T=
           \widetilde{\mathbf A}\widetilde{\mathbf A}^T$
     	    it follows that
     	    \begin{equation*}
     	           \begin{split}
     	          &\ \mathbf A\Diag{\widehat{\lambda}_1,\dots,\widehat{\lambda}_R}\mathbf A^T\\
     	          =&\ 
     	           \mathbf A\Diag{\lambda_1,\dots,\lambda_R}\mathbf A^T
     	           +\mathbf A_{i_1,\dots,i_{m-1-k}} \Diag{\mu_1^2,\dots,\mu_{m-1-k}^2}\mathbf A_{i_1,\dots,i_{m-1-k}}^T\\
     	           =&\ \widetilde{\mathbf A}\widetilde{\mathbf A}^T+ \mathbf A_{i_1,\dots,i_{m-1-k}} \Diag{\mu_1^2,\dots,\mu_{m-1-k}^2}\mathbf A_{i_1,\dots,i_{m-1-k}}^T\\
     	           =&\ [\mathbf S\ \mathbf A_{i_1,\dots,i_{m-1-k}}\Diag{\mu_1,\dots,\mu_{m-1-k}}][\mathbf S\ \mathbf A_{i_1,\dots,i_{m-1-k}}\Diag{\mu_1,\dots,\mu_{m-1-k}}]^T\\
     	           =&\ \widehat{\mathbf A}\widehat{\mathbf A}^T.
     	           \end{split}
     	  \end{equation*}
     	          Thus,  $(\mathbf A,\mathbf C,\mathbf L,\mathbf D,\mathbf U, \widehat{\lambda}_1,\dots, \widehat{\lambda}_R, \widehat{\mathbf A})\in Z_{1,\dots,I}$.
     	          
  \qquad\textup{(3iii)} By Lemma \ref{Theorem:FDTh} and \eqref{indscal1fromintro2}, 
 \begin{equation*}
 \begin{split}
  \dim \overline{\pi(Z_{1,\dots,I})}\leq&\ \dim Z_{1,\dots,I}-d\\
  =&\ IR+KR+R-\frac{I^2-I(2m-3)}{2}-
  \frac{(m-1)(m-2)}{2}-1\\
  \leq&\ IR+KR-1.
  \end{split}
  \end{equation*}
   \qquad\textup{(4)} Hence, $\dim W\leq\max\limits_{1\leq u_1<\dots<u_I\leq R}\dim \pi(Z_{u_1,\dots,u_I}) \leq IR+KR-1$.
   
   Since $W\subset \mathbb C^{I\times R}\times \mathbb C^{K\times R}$ and 
   $\dim W\leq IR+KR-1$, it follows from  Lemma \ref{lemma1} that $\mu_{\mathbb C}(W)=0$. 
\qquad\endproof

 The following lemma is well-known, it formalizes for a particular case the fact that $\dim Z_{\widehat{\mathbf A}}$ coincides with the number of ``free parameters'' (namely, one parameter for $\alpha$ and
 $\frac{(m-1)(m-2)}{2}$ parameters for the complex orthogonal matrix $\mathbf T$) used to describe $Z_{\widehat{\mathbf A}}$.
 \begin{lemma} \label{lemma:moregencomport}
 Let the $I\times(m-1)$ matrix $\widehat{\mathbf A}$ have full column rank and
$$
Z_{\widehat{\mathbf A}}=\{\alpha\widehat{\mathbf A}\mathbf T :\ \alpha\in\mathbb C, \ \mathbf T\in\mathbb C^{(m-1)\times(m-1)},\  \mathbf T\mathbf T^T=\mathbf I\}\subset\mathbb C^{I\times(m-1)}.
$$
 Then $\dim Z = \frac{(m-1)(m-2)}{2}+1$.
\end{lemma} 

       \section{Proof of Proposition \ref{Proposition1.4}}\label{Sectionformercor}
Proposition  \ref{Proposition1.4} follows immediately from  Proposition \ref{PropositionINDSCAL}. To formalize the derivation, we need the trivial observation expressed in  the following lemma.
\begin{lemma}\label{lemma:trivial}
Assume that the  SFS-CPD of  an $I\times I\times K_1$ SFS-tensor of SFS-rank $R_1$ is generically unique and let
$K_1\leq K_2$, $R_2\leq R_1$. Then 
 the  SFS-CPD of  an $I\times I\times K_2$ SFS-tensor of SFS-rank $R_2$ is generically unique.
\end{lemma}

{\em Proof of Proposition \ref{Proposition1.4}.}
Assume that $4\leq I\leq K\leq \frac{I^2-I}{2}$. By Proposition \ref{PropositionINDSCAL},
the  SFS-CPD of  an $I\times I\times K$ SFS-tensor of SFS-rank $R_1=K$ is generically unique.
The uniqueness for $R=R_2\leq R_1=K$ follows from Lemma \ref{lemma:trivial}.

Assume next that $K\geq K_1=\frac{I^2-I}{2}$. By Proposition \ref{PropositionINDSCAL},
the  SFS-CPD of  an $I\times I\times K_1$ SFS-tensor of SFS-rank $R_1=\frac{I^2-I}{2}$ is generically unique.
Now the result again follows from Lemma \ref{lemma:trivial},  for $K_2=K$ and $R_2=R$.
\qquad\endproof
\section{Uniqueness of the SFS-CPD and proof of Proposition \ref{PropositionINDSCAL2}}\label{Section5}
The following deterministic  condition for  uniqueness of the SFS-CPD will be checked for a generic SFS-tensor in the
proof of Proposition \ref{PropositionINDSCAL2}.
\begin{proposition}\label{Prop:Indscal2}
Let $\mathcal T=[\mathbf A, \mathbf A, \mathbf C]_R$ and
$m_{\mathbf A}:=R-r_{\mathbf A}+2$. Assume that
\begin{itemize}
\item[\textup{(i)}] $\max(\min( k_{\mathbf A}, k_{\mathbf C}-1),\ \min( k_{\mathbf A}-1, k_{\mathbf C}))+k_{\mathbf A}\geq R+1$;
\item[\textup{(ii)}] condition \condUmA\ holds for the pair $(\mathbf A,\mathbf C)$.
\end{itemize}
Then $r_{\mathcal T}=R$ and the SFS-CPD of tensor $\mathcal T$  is unique.
\end{proposition}
\begin{proof}
It is clear that if the CPD of the reshaped $I\times K\times I$ tensor $\widetilde{\mathcal T}=[\mathbf A, \mathbf C, \mathbf A]_R$ is unique, then the CPD, and in particular, SFS-CPD $\mathcal T=[\mathbf A, \mathbf A, \mathbf C]_R$ is unique.
We show that \textup{(i)}--\textup{(iii)} in Proposition \ref{ProFullUniq1onematrixWm} hold for
$\widetilde{\mathcal T}=[\mathbf A, \mathbf C, \mathbf A]_R$: \textup{(i)} is obtained by replacing $k_{\mathbf B}$ and $k_{\mathbf C}$ by $k_{\mathbf C}$ and $k_{\mathbf A}$, respectively and \textup{(ii)} and \textup{(iii)} follow from \eqref{eq:Umimplies1} and \eqref{eq:Umimplies2}, respectively.
\end{proof}

{\em Proof of Proposition \ref{PropositionINDSCAL2}.} We show that
\begin{equation*} 
\mu_{\mathbb F}\{(\mathbf A,\mathbf C):\ \text{\textup{(i)} or \textup{(ii)} of Proposition \ref{Prop:Indscal2} does not hold}\}=0.
\end{equation*}
As in the proofs of Propositions \ref{Propositionunstructured} and \ref{PropositionINDSCAL}, it is sufficient to prove that
\begin{equation*} 
\begin{split}
\mu_{\mathbb F}\{(\mathbf A,\mathbf C):\ &\text{\textup{(i)} or \textup{(ii)} of Proposition \ref{Prop:Indscal2} does not hold and}\\
&k_{\mathbf A}=I,\ k_{\mathbf C}=K \}=0.
\end{split}
\end{equation*}

{\bf Condition \textup{(i)}:}
By \eqref{eqIndscalKlessI1}--\eqref{eqIndscalKlessI},
$$
R+2\leq I+K+\frac{I+K-3-\sqrt{(I+K-3)^2+8K-12}}{2}\leq I+K
$$
which easily implies that condition \textup{(i)} of Proposition \ref{Prop:Indscal2}
holds for all matrices $\mathbf A$ and $\mathbf C$ such that 
$k_{\mathbf A}=I$ and $k_{\mathbf C}=K$.

{\bf Condition \textup{(ii)}:}
Let
\begin{equation} \label{eq:fullkrank2787822}
\begin{split}
W=\{(\mathbf A,\mathbf C):\ \text{\textup{(ii)} of Proposition \ref{Prop:Indscal2} } & \text{does not hold and}\\
& k_{\mathbf A}=I, \ k_{\mathbf C}=K \}\subset \mathbb F^{I\times R}\times \mathbb F^{K\times R}.
\end{split}
\end{equation}
As in the proofs of Propositions \ref{Propositionunstructured} and \ref{PropositionINDSCAL}, to prove that $\mu_{\mathbb F}\{W\}=0$   we show that
for $\mathbb F=\mathbb C$, 
\begin{equation} \label{eq:fullkrank2344}
\dim \overline{W}\leq IR+KR-1.
\end{equation}
To prove bound \eqref{eq:fullkrank2344} we follow the steps in Procedure \ref{proc}.

{\bf Phase I: ``Parameterization''.}

\textup{(1)} We associate $W$ with a certain $\pi(\widehat Z)$.
By Definition \ref{Def:Um},  condition \condUm\  does not hold for the pair $(\mathbf A, \mathbf C)$ if and only if
there exist values  $\lambda_1,\dots,\lambda_R$, and matrices $\widetilde{\mathbf A}\in \mathbb C^{I\times (m-1)}$, $\widetilde{\mathbf C}\in \mathbb C^{K\times (m-1)}$ such that
$$
 \mathbf A\Diag{\lambda_1,\dots,\lambda_R}\mathbf C^T=\widetilde{\mathbf A}\widetilde{\mathbf C}^T, \ \
$$
but $\omega(\lambda_1,\dots,\lambda_R)\geq m$.
We claim that if additionally
$k_{\mathbf A}=I$ and $k_{\mathbf C}=K$, then
$\omega(\lambda_1,\dots,\lambda_R)\geq K$. Indeed, if
$m=m_{\mathbf A}\leq \omega(\lambda_1,\dots,\lambda_R) <K$, then  by the Frobenius inequality,
\begin{equation*}
\begin{split}
m-1\geq&\ r_{\widetilde{\mathbf A}\widetilde{\mathbf C}^T}=r_{\mathbf A\Diag{\lambda_1,\dots,\lambda_R}\mathbf C^T}\\
\geq&\
r_{\mathbf A\Diag{\lambda_1,\dots,\lambda_R}}+r_{\Diag{\lambda_1,\dots,\lambda_R}\mathbf C^T}-
r_{\Diag{\lambda_1,\dots,\lambda_R}}\\
=&\ \omega(\lambda_1,\dots,\lambda_R)+\omega(\lambda_1,\dots,\lambda_R)-\omega(\lambda_1,\dots,\lambda_R)=
\omega(\lambda_1,\dots,\lambda_R)\geq m,
\end{split}
\end{equation*}
which is a contradiction.
Hence, $W$ in \eqref{eq:fullkrank2787822} can be expressed as
\begin{align}
W=\{(\mathbf A,\mathbf C):\ &\text{\condUm\  does not hold for the pair } (\mathbf A,\mathbf C),\
 k_{\mathbf A}=I, \ k_{\mathbf C}=K \} \nonumber\\
=\{(\mathbf A, \mathbf C):\ &\text{there exist } \lambda_1,\dots,\lambda_R\in\mathbb C,\ \widetilde{\mathbf A}\in \mathbb C^{I\times (m-1)},\ \text{ and }\ \widetilde{\mathbf C}\in \mathbb C^{K\times (m-1)}\ \nonumber \\
&\text{ such that }\  \mathbf A\Diag{\lambda_1,\dots,\lambda_R}\mathbf C^T=\widetilde{\mathbf A}\widetilde{\mathbf C}^T,
\label{eq:ABC1last}\\
&\ \omega(\lambda_1,\dots,\lambda_R)\geq K,\ k_{\mathbf A}=I,\ k_{\mathbf C}=K\}.\label{eq:ABC3last}
\end{align}
It is now clear that $W=\pi(\widehat Z)$, where 
$$
\widehat Z=\{(\mathbf A, \mathbf C,
 \lambda_1,\dots,\lambda_R,
 \widetilde{\mathbf A}, \widetilde{\mathbf C}):\ \text{ \eqref{eq:ABC1last}--\eqref{eq:ABC3last} hold}\}
$$
is a subset of
$ 
\mathbb C^{I\times R}\times \mathbb C^{K\times R}\times\mathbb C^R\times \mathbb C^{I\times (m-1)}\times \mathbb C^{K\times (m-1)}
$ 
and $\pi$ is the projection onto the first two factors 
\begin{equation*} 
\pi:\ \mathbb C^{I\times R}\times \mathbb C^{K\times R}\times\mathbb C^R\times \mathbb C^{I\times (m-1)}\times \mathbb C^{K\times (m-1)}\rightarrow
 \mathbb C^{I\times R}\times  \mathbb C^{K\times R}.
\end{equation*}

\textup{(2)}  Since
\begin{equation*}
\omega(\lambda_1,\dots,\lambda_R)\geq K\ \ \Leftrightarrow\ \ 
\lambda_{u_1}\cdots\lambda_{u_K}\ne 0\ \text{ for some }\ 1\leq u_1<\dots<u_K\leq R 
\end{equation*}
we obtain
\begin{equation*}
\begin{split}
\widehat Z= \bigcup\limits_{1\leq u_1<\dots<u_K\leq R}\{(\mathbf A, \mathbf C,
 \lambda_1,\dots,\lambda_R,
 \widetilde{\mathbf A}, \widetilde{\mathbf C}):\ &\text{ \eqref{eq:ABC1last} holds,}\\ 
 \ &\lambda_{u_1}\cdots\lambda_{u_K}\ne 0,\ k_{\mathbf A}=I,\ k_{\mathbf C}=K\}.
 \end{split}
 \end{equation*}
  Let $\mathbf C_{u_1,\dots,u_K}$ denote the submatrix of  $\mathbf C$  formed by columns $u_1,\dots,u_K$.
  Dropping the condition $k_{\mathbf A}=I$ and relaxing the condition $k_{\mathbf C}=K$ as $\det{\mathbf C_{u_1,\dots,u_K}}\ne 0$ we obtain that
 $\widehat Z\subset\bigcup\limits_{1\leq u_1<\dots<u_K\leq R} Z_{u_1,\dots,u_K}$, where
 \begin{equation*}
 \begin{split}
 Z_{u_1,\dots,u_K}:=\{(\mathbf A, \mathbf C,
  \lambda_1,\dots,\lambda_R,
  \widetilde{\mathbf A}, \widetilde{\mathbf C}):\ 
  &\text{ \eqref{eq:ABC1last} holds},\\   
  & \lambda_{u_1}\cdots\lambda_{u_K}\ne 0, \ \det{\mathbf C_{u_1,\dots,u_K}}\ne 0\}.
  \end{split}
  \end{equation*}
 We show that all $Z_{u_1,\dots,u_K}$ are of the form \eqref{eq:setZ}.
  To simplify the presentation we restrict ourselves to the case $(u_1,\dots, u_K)=(1,\dots,K)$. The general case can be proved in the same way.
  Let the matrices $\mathbf A$ and $\mathbf C$ be partitioned as 
   $$
   \mathbf A=[\underbracket{\ \bar{\mathbf A}\ }_{K}\ \underbracket{\ \bar{\bar{\mathbf A}}\ }_{R-K}],\qquad
   \mathbf C=[\underbracket{\ \bar{\mathbf C}\ }_{K}\ \underbracket{\ \bar{\bar{\mathbf C}}\ }_{R-K}],
   $$
   so that $\bar{\mathbf C}=\mathbf C_{1,\dots,K}$. By \eqref{eq:ABC1last},
\begin{equation*} 
\bar{\mathbf A}=\left( \widetilde{\mathbf A}\widetilde{\mathbf C}^T - 
\bar{\bar{\mathbf A}}\Diag{\lambda_{K+1},\dots,\lambda_R}\bar{\bar{\mathbf C}}^T\right)
\left(\Diag{\lambda_{1},\dots,\lambda_K}\bar{\mathbf C}\right)^{-1}.
\end{equation*}
Thus, by Cramer's rule, each entry of $\bar{\mathbf A}$   can be written as
     a ratio of two polynomials in  entries of $\bar{\bar{\mathbf A}}$, $\mathbf C$, $\widetilde{\mathbf A}$, $\widetilde{\mathbf C}$ and the values  $\lambda_1,\dots,\lambda_R$. By the assumptions
      $\lambda_1\cdots\lambda_K\ne 0$ and $\det\bar{\mathbf C}\ne 0$, the denominator polynomial is nonzero.
   
    Hence, $Z_{1,\dots,K}$
 is indeed of the form \eqref{eq:setZ}, where  $z_1,\dots,z_m$ correspond to the entries of
  $\bar{\bar{\mathbf A}}$, $\mathbf C$, $\widetilde{\mathbf A}$, $\widetilde{\mathbf C}$ and the values  $\lambda_1,\dots,\lambda_R$, and where $z_{m+1},\dots,z_n$  correspond to the entries of $\bar{\mathbf A}$.\\
 {\bf Phase II: Obtaining a bound on $\dim W$.}
 
 \textup{(3)} 
  To obtain bounds on $\dim \pi(Z_{u_1,\dots,u_K})$ we follow the steps in Procedure \ref{procaux} for $X=\pi(Z_{u_1,\dots,u_K})$. W.l.o.g.   we again restrict ourselves to the case $(u_1,\dots,u_K)=(1.\dots,K)$.
  
  \textup{(3i)} By Lemma \ref{lemma:2.8}, $\dim Z_{1,\dots,K}= I(R-K)+KR+I(m-1)+K(m-1)+R$.
  
  \textup{(3ii)}  From the definition of $Z_{1,\dots,K}$ it follows that if 
    $(\mathbf A, \mathbf C,
        \lambda_1,\dots,\lambda_R, \widetilde{\mathbf A}, \widetilde{\mathbf C})\in Z_{1,\dots,K}$, then
     $(\mathbf A, \mathbf C,
       \alpha\lambda_1,\dots,\alpha\lambda_R,
       \widehat{\mathbf A}\mathbf T, \alpha\widehat{\mathbf C}\mathbf T^{-T})\in Z_{1,\dots,K}$, where $\alpha$ is an arbitrary nonzero value,  $\widehat{\mathbf A}$ is an arbitrary full column rank matrix such that $\textup{range}(\widehat{\mathbf A})\supset \textup{range}(\widetilde{\mathbf A})$,
           $\mathbf T$ is an arbitrary nonsingular   $(m-1)\times (m-1)$ matrix, and $\widehat{\mathbf C}$ satisfies  $\widehat{\mathbf A}\widehat{\mathbf C}^T=\widetilde{\mathbf A}\widetilde{\mathbf C}^T$.
        Hence,
      $$
      f^{-1}(\mathbf A, \mathbf C)\supset\{\mathbf A, \mathbf C,
       \alpha\lambda_1,\dots,\alpha\lambda_R,
       \widehat{\mathbf A}\mathbf T, \alpha\widehat{\mathbf C}\mathbf T^{-T}:\ \alpha\ne 0,\ \det\mathbf T\ne 0\}.
       $$
        By Lemma \ref{lemma:propalggeom} \textup{(iii)},
         $
            \dim f^{-1}(\mathbf A, \mathbf C)\geq \dim \{\alpha\lambda_1, \widehat{\mathbf A}\mathbf T:\ \alpha\ne 0,\ \det\mathbf T\ne 0\}
           $. Since by construction, the matrix $\widehat{\mathbf A}$ has full column rank and by assumption $\lambda_1\ne 0$, it follows that
         $	
         \dim \{\alpha\lambda_1, \widehat{\mathbf A}\mathbf T:\ \alpha\ne 0,\ \det\mathbf T\ne 0\}=1+(m-1)^2=:d.
         $
         Thus, $\dim f^{-1}(\mathbf A, \mathbf C)\geq d$.

 \textup{(3iii)}  By Lemma \ref{Theorem:FDTh} and 
 \eqref{eq3new222intro}, 
  \begin{equation*}
  \begin{split}
  \dim \overline{\pi(Z_{1,\dots,K})}\leq&\ \dim Z_{1,\dots,K}-d\\
 =&\ I(R-K)+KR+I(m-1)+K(m-1)+R-1-(m-1)^2\\
 \leq&\ IR+KR-1.
 \end{split}
  \end{equation*}
 \qquad \textup{(4)} Hence, $\dim W\leq\max\limits_{1\leq u_1<\dots<u_K\leq R}\dim \pi(Z_{u_1,\dots,u_K})\leq IR+KR-1$. 
         
Since $W\subset \mathbb C^{I\times R}\times \mathbb C^{K\times R}$ and 
$\dim W\leq IR+KR-1$, it follows from  Lemma \ref{lemma1} that $\mu_{\mathbb C}(W)=0$.
\qquad\endproof
\section{Conclusion}
We have obtained new conditions
guaranteeing generic uniqueness of a CPD and INDSCAL. The overall derivation was based on deterministic conditions for uniqueness previously
obtained in \cite{PartII}. Our bounds improve existing results in the case when one of the tensor's dimensions is significantly
larger than  the other dimensions. The derivation made use of algebraic geometry. We summarized some results from algebraic geometry in a procedure  that may be used in different applications to study generic conditions.  Important is that the condition can be  represented in terms of a subset of $\mathbb C^n$ parametrized by rational functions. The specific realization of the procedure depends on the problem at hand. In this paper we have explained how the procedure works for  CPD and INDSCAL. In future work, we will consider the generic uniqueness of block term decompositions \cite{LDLBTPartII}.

{\bf Acknowledgments:} The authors wish to thank Giorgio Ottaviani, Ed Dewey, and Galyna Dobrovolska for their 
assistance in algebraic geometry.

\bibliographystyle{plain}
\bibliography{assignment2}
\end{document}